\documentclass{amsart}

\usepackage{amsfonts,amsmath,amssymb}
\usepackage{latexsym,amscd,graphicx}

\newcommand{\bg}{{\overline{g}}}

\newcommand{\bd}{{\overline{d}}}
\newcommand{\bz}{{\overline{z}}}

\newcommand{\bx}{{\overline{x}}}
\newcommand{\by}{{\overline{y}}}
\newcommand{\bch}{{\overline{\chi}}}

\newcommand{\wx}{{\widetilde{x}}}
\newcommand{\wy}{{\widetilde{y}}}

\newcommand{\bS}{{\overline{S}}}

\newcommand{\wh}[1]{\widehat{#1}}
\newcommand{\wt}[1]{\widetilde{#1}}

\newcommand{\ve}{\varepsilon}
\newcommand{\vp}{\varphi}
\newcommand{\ld}{\ldots}

\DeclareMathOperator*{\ot}{\otimes}
\DeclareMathOperator*{\op}{\oplus}

\newcommand{\beq}{\begin{equation}}
\newcommand{\eeq}{\end{equation}}
\newcommand{\beas}{\begin{eqnarray*}}
\newcommand{\eeas}{\end{eqnarray*}}
\newcommand{\id}{\mathrm{id}}

\newcommand{\cC}{\mathcal{C}}
\newcommand{\cD}{\mathcal{D}}
\newcommand{\cE}{\mathcal{E}}

\newcommand{\cP}{\mathcal{P}}

\newcommand{\R}{\mathbb{R}}
\newcommand{\C}{\mathbb{C}}
\renewcommand{\H}{\mathbb{H}}
\newcommand{\zz}{\mathbb{Z}_2\times\mathbb{Z}_2 }
\newcommand{\z}{\mathbb{Z}_2}

\newcommand{\ZZ}{\mathbb{Z}}
\newcommand{\bZ}{{\mathbb Z}}

\newcommand{\FF}{\mathbb{F}}

\DeclareMathOperator{\End}{\mathrm{End}}

\DeclareMathOperator{\alg}{\mathrm{alg}}

\DeclareMathOperator{\Ker}{\mathrm{Ker}\,}
\DeclareMathOperator{\rank}{\mathrm{rank}\,}
\DeclareMathOperator{\Aut}{\mathrm{Aut}}

\DeclareMathOperator{\Supp}{\mathrm{Supp}}

\newcommand{\GL}{\mathrm{GL}}

\newtheorem{theorem}{Theorem}[section]

\newtheorem{lemma}[theorem]{Lemma}

\newtheorem{proposition}[theorem]{Proposition}

\newtheorem{remark}[theorem]{Remark}

\begin{document}

\title{Graded Division Algebras over the Field of Real Numbers}

\author[Bahturin]{Yuri Bahturin}
\address{Department of Mathematics and Statistics, Memorial
University of Newfoundland, St. John's, NL, A1C5S7, Canada}
\email{bahturin@mun.ca}

\author[Zaicev]{Mikhail Zaicev}
\address{Faculty of Mathematics and Mechanics, Moscow State University, Russia}
\email{zaicevmv@mail.ru}

\thanks{The first author acknowledges support by by NSERC grant \# 227060-14. The second author acknowledges support by Russian Science Foundation grant \# 16-11-10013}

\begin{abstract}
We give a full classification, up to equivalence, of finite-dimensional graded division algebras over the field of real numbers. The grading group is any abelian group. 
\end{abstract}

\maketitle

\section{Introduction} In this paper we will deal only with finite-dimensional algebras over a field $F$, which will be either the field $\R$ of real numbers or the field $\C$ of complex numbers. A unital algebra $R$ over a field $F$ graded by a group $G$ is called \textit{graded division} if every nonzero homogeneous element is invertible. Each such algebra is graded simple, that is, $R$ has no proper nonzero graded ideals. As an ungraded algebra, a graded division algebra does not need to be simple, as shown by the basic example of the group algebra $FG$. But it is known (see, e.g. \cite{BZ11}) that graded division algebras are semisimple, that is, isomorphic to the sum of one or more simple algebras.  According to the graded analogues of Schur's Lemma and Density Theorem (see, for example, \cite{NVO} or \cite{BMPZ} or \cite{EK}) any finite-dimensional graded simple algebra is isomorphic to the algebra $\End_DV$ of endomorphisms of a finite-dimensional graded (right) vector space over a graded division algebra $D$. If, additionally, $R$ is simple, it is obvious that $D$ must be simple, as well.

In the case where the field $F$ is algebraically closed, all simple graded division algebras have been described in \cite{BSZ} and \cite{BZ02}. For full account see \cite[Chapter 1]{EK}, where the authors treat also the case of Artinian algebras. In \cite{BBZ} (see also \cite{BZ10}, for a particular case) the authors treat the case of graded primitive algebras with minimal one-sided graded ideals. If such algebras are locally finite, the graded division algebras arising by graded Schur's Lemma, are finite-dimensional and so the description provided in the case of finite-dimensional algebras works in this situation, as well.

In our recent paper \cite{BZGDS} we have classified division gradings on simple real finite-dimensional algebras, up to equivalence. In \cite{ARE} the author provided another approach to  the classification of division gradings on these algebras, also up to isomorphism. 

In the present paper we classify all finite-dimensional real graded division algebras. This is done case by case, depending on various factors. In Theorem \ref{Pauli} we deal with algebras endowed with so called Pauli or Sylvester gradings. They come from complex graded division algebras, by restriction of the scalars to the field of real numbers. Theorem \ref{tcca} classifies,  up to equivalence, all finite-dimensional real commutative graded division algebras. In Theorem \ref{t_2_1_m} we classify, up to equivalence, all finite-dimensional real noncommutative graded division algebras with 1-dimensional homogeneous components. In Theorem \ref{teq_2_dim} we classify,  up to equivalence, all finite-dimensional real graded division algebras whose homogeneous components are 2-dimensional, as in the case of Pauli gradings, but the identity component is not central.  Finally,  in Theorem \ref{tMain} we do the same in the case of  algebras with  4-dimensional components. This covers all possible cases.

Our proofs in this paper are not an extension of the arguments in \cite{BZGDS} but a somewhat different argument, which can simplify the approach taken in  \cite{BZGDS}. 

A relevant paper is \cite{ABFP}, where the authors describe not necessarily associative graded simple algebras, under certain additional conditions.

\section{Preliminaries}\label{sPre}

A vector space decomposition $\Gamma: V=\bigoplus_{g\in G}V_g$ is called a grading of a vector space $V$ over a field $F$ by a set $G$. The subset $S$ of all $s\in G$ such that $V_s\neq \{ 0\}$ is called the \textit{support} of $\Gamma$ and is denoted by $\Supp\Gamma$ (also as $\Supp V$, if $V$ is endowed just by one grading). If $\Gamma': V'=\bigoplus_{g'\in G'}V'_{g'}$ is a grading of another space, then a homomorphism of gradings $\vp:\Gamma\to\Gamma'$ is a linear map $f:V\to V'$ such that for each element $g\in G$ there exists an element $g'\in G'$ such that $\vp(V_g)\subset V_{g'}$. If $\vp$ has an inverse as a homomorphism of gradings, then we say that $\vp:\Gamma\to\Gamma'$  is an \textit{equivalence} of gradings $\Gamma$ and $\Gamma'$ (or graded vector spaces $V$ and $V'$).

A grading $\Gamma: R=\bigoplus_{g\in G}R_g$ of an algebra $R$ over a field $F$ is an \textit{algebra} grading if for any $s_1,s_2\in S=\Supp\Gamma$ such that $R_{s_1}R_{s_2}\ne \{0\}$ there is $s_3\in G$ such that $R_{s_1}R_{s_2}\subset R_{s_3}$. Two algebra gradings $\Gamma: R=\bigoplus_{g\in G}R_g$ and $\Gamma': R'=\bigoplus_{g'\in G'}R'_{g'}$ of algebras over a field $\FF$ are called \textit{equivalent} if there exist an algebra isomorphism $\vp:R\to R'$, which is an equivalence of vector space gradings. In this case there is a bijection $\alpha:\Supp\Gamma\to\Supp\Gamma'$ such that $f(R_g)=R'_{\alpha(g)}$. The group given by the set of generators $S$ and defining relations $s_1s_2=s_3$ provided that $\{ 0\}\ne R_{s_1}R_{s_2}\subset R_{s_3}$, is called the \textit{universal group} of the grading $\Gamma$ and denoted by $U(\Gamma)$.

If $G$ is a group then a grading $\Gamma: R=\bigoplus_{g\in G}R_g$ of an algebra $R$ over a field $F$ is called a \textit{group grading} if for any $g,h\in G$, we have $R_gR_h\subset R_{gh}$. Normally, it is assumed that the grading group $G$ is generated by $\Supp\Gamma$. If $\vp:\Gamma\to \Gamma':R\to R'$ is an equivalence of gradings of algebras $R$ and $R'$ by groups $G$ and $G'$ and the accompanying bijection $\alpha:\Supp\Gamma\to\Supp\Gamma'$ comes from an isomorphism of groups $\alpha:G\to G'$ then we call $\vp$ a \textit{weak isomorphism} and say that $\Gamma$ and $\Gamma'$ (also $R$ and $R'$) are \textit{weakly isomorphic}. Finally, if $G=G'$ and $\alpha=\id_G$ then $\Gamma$ and $\Gamma'$ are called \textit{ isomorphic}.

Note that if, say, $\Gamma$ is a \textit{strong grading}, that is, $R_gR_h=R_{gh}$, for any $g,h\in G$, then $\Gamma$ and $\Gamma'$ are equivalent if and only if they are weakly isomorphic.

If $\Gamma: R=\bigoplus_{s\in S}R_s$ and $\Gamma': R=\bigoplus_{s'\in S'}R_{s'}^\prime$ are two gradings of the same algebra labeled by the sets $S$ and $S'$ then we say that $\Gamma$ is a \textit{refinement} of $\Gamma'$ if for any $s\in S$ there is $s'\in S'$ such that $R_s\subset R_{s'}$. We also say that $\Gamma'$ is a \textit{coarsening} of $\Gamma$. The refinement $\Gamma$ is proper if for at least one $s$ the containment $R_s\subset R_{s'}$ is proper. A grading which does not admit proper refinements is called \textit{fine}. Assume $\Gamma,\Gamma'$ are group gradings, so that $G'=G/T$. If $R^\prime_{\bg}=\bigoplus_{g\in\bg}R_g$, for all $\bg\in G'$, then $\Gamma'$ is a coarsening of $\Gamma$ called \textit{factor-grading}. In the case of complex gradings all division gradings are fine, while in the case of real numbers this is no longer true (see examples in Section \ref{ssgda} below).

\subsection{Tensor products of division gradings}\label{sstp}

Given groups $G_1, G_2,\ld,G_m$ and $G_k$-graded algebras $R_1, R_2,\ld, R_m$, $k=1,\ld,m$, one can endow the tensor product of algebras $R=R_1\ot R_2\ot\cdots\ot R_m$ by a $G=G_1\times G_2\times\cdots\times G_m$-grading, called the \textit{tensor product of gradings} if one sets
\begin{displaymath}
R_{(g_1,g_2,\ld,g_m)}=(R_1)_{g_1}\ot (R_2)_{g_2}\ot\cdots\ot (R_m)_{g_m}.
\end{displaymath}
Here $g_k\in G_k$, for all $k=1,2,\ldots,m$.

In the case of division algebras over an algebraically closed field, the tensor product of two graded division algebras is a graded division algebra. This is no longer true in our case. Indeed, $D_1\otimes D_2$, where $D_i=\R,\C,\H$, is a division algebra if and only if at least one of $D_i$ is $\R$. If $R$ is a $G$-graded division algebra, $S$ an $H$-graded division algebra, $(R\otimes S) _{(g,h)}=R_g\ot S_h$, for all $(g,h)\in G\times H$. Clearly, all nonzero elements in each homogeneous component are invertible if this is true for the identity component. As a result, the tensor product of two division gradings is a division grading only if the identity component of at least one of them is one-dimensional.

\subsection{Basic properties of division gradings}\label{ssBP}
We start with fixing few well-known useful properties (see e.g. \cite[Chapter 2]{EK}).
\begin{lemma}\label{lsf} Let $\Gamma: R=\bigoplus_{g\in G} R_g$ be a grading by a group $G$ on an (associative) algebra $R$ over a field $F$.  If $\Gamma$ is a division grading then the following hold.
\begin{enumerate}
\item The identity component $R_e$ of $\Gamma$ is a division algebra over $F$;
\item Given $g\in G$ and a nonzero $a\in R_g$, we have $R_g=aR_e$;
\item For  any $g\in G$, $\dim R_g=\dim R_e$ and $\dim R=|\Supp\Gamma|\dim R_e$;
\item $\Supp(\Gamma)$ is a subgroup of $G$ isomorphic to the universal group $U(\Gamma)$.\hfill$\square$
\end{enumerate}
\end{lemma}

Since our base field is $\R$, it follows that $R_e$ is one of $\R$, $\C$, or $\H$,  the division algebra of quaternions.

As mentioned above, the support of the division grading is a subgroup in the grading group. This makes it natural to always assume that the support of $R$ equals the whole of $G$.  Thus, when we speak about gradings on finite-dimensional division graded algebras, we may assume that the grading group $G$ is finite.

One notational remark. Given an element $g$ of order $n$ in a group $G$, we denote by $(g)_n$ the cyclic subgroup generated by $g$. Given vectors $v_1,\ld,v_m$ in a real vector space $V$, we denote by $\langle v_1,\ld,v_m\rangle$ the linear span of $v_1,\ld,v_m$, with coefficients in $\R$. To avoid confusion with number $1\in\R$, we will denote the identity element of a graded division algebra $R$ by $I$.

\section{Simple graded division algebras}\label{ssgda}

We recall some low dimensional simple graded division algebras and the classification of simple graded division algebras, up to equivalence (see \cite{BZGDS} and \cite{ARE}).

The simplest examples of graded division algebras are $\R$, $\C$ and the quaternions $\H$. These are graded by the trivial group. Given any group $G$, the group algebras $\R G$, $\C G$ and $\H G=\H\ot\R G$ are further examples of real graded division algebras, of dimensions $|G|$, $2|G|$ and $4|G|$, respectively. Also, we can refine the trivial gradings on  $\C$  and $\H$ to obtain the $\z$-gradings  $\C^{(2)}$ and $\H^{(2)},$ as follows. If we set $\z\cong(\alpha)_2$, then
\begin{equation}\label{eC2H2}
\C^{(2)}=\langle 1\rangle_e\oplus \langle i\rangle_\alpha,\; \H^{(2)}=\langle 1,i\rangle_e\oplus \langle j,k\rangle_\alpha.
\end{equation}
Also, if we set $\zz\cong(\alpha)_2\times (\beta)_2$ then a $\zz$-refinement, $\H^{(4)}$, on $\H$ will look like the following:
\begin{equation}\label{eH4}
\H^{(4)}=\langle 1\rangle_e\oplus \langle i\rangle_\alpha\oplus\langle j\rangle_\beta\oplus \langle k\rangle_{\alpha\beta}.
\end{equation}
The matrix algebra $M_2=M_2(\R)$, with trivial grading, is not a graded division algebra, but it has $\z$- and $\z^2$-refinements $M_2^{(2)}$ and $M_2^{(4)}$, which are graded division algebras. Recall the Pauli (Sylvester) matrices
\begin{equation}\label{eE1}
A=\left(\begin{array}{cc} 1 &0\\0&-1\end{array}\right)\; B=\left(\begin{array}{cc} 0 &1\\1&0\end{array}\right)\;C=\left(\begin{array}{cc} 0 &1\\-1&0\end{array}\right).
\end{equation}
Then $A^2=B^2=I$, $AB=-BA=C$. It follows then that $C^2=-I$, $AC=-CA=B$, $BC=-CB=-A$. So we have the following
\begin{equation}\label{eM2M4}
M_2^{(2)}=\langle I,C\rangle_e\oplus \langle A,B\rangle_\alpha,\;M_2^{(4)}=\langle I\rangle_e\oplus \langle C\rangle_\alpha\oplus\langle A\rangle_\beta\oplus \langle B\rangle_{\alpha\beta}.
\end{equation}

The only 8-dimensional real simple algebra is $R=M_2(\C)$. Since $M_2(\R)\ot\C\cong M_2(\C)\cong \H\ot\C$, we can obtain division gradings on $R$ in various ways. A division $\z$-grading can be obtained as $ \H\ot\C^{(2)}$. A division $\z^2$-grading on $R$  can be obtained as  $ \H^{(4)}\ot\C$, $M_2^{(4)}\ot\C$, $\H^{(2)}\ot\C^{(2)}$ and $M_2^{(2)}\ot\C^{(2)}$. It is known from \cite{BZGDS} that the first two gradings are equivalent, and also the last two are equivalent. A division $\z^3$-grading on $R$ can be obtained in two equivalent ways: $ \H^{(4)}\ot\C^{(2)}$ or $M_2^{(4)}\ot\C^{(2)}$. Notice that the natural isomorphism $\vp:\H\ot\C\to M_2(\R)\ot\C$ defined by $\vp(i\ot z)=A\ot iz$, $\vp(j\ot z)=B\ot iz$, for any $z\in\C$, induces a graded isomorphisms for the refinements $\vp:H^{(2)}\ot\C^{(2)}\to
M_2^{(2)}\ot\C^{(2)}$ and $\vp:H^{(4)}\ot\C^{(2)}\to
M_2^{(4)}\ot\C^{(2)}$.

So far, all division gradings on $M_2(\C)$ have appeared as tensor product gradings. However, there are gradings on this algebra, which are not tensor products. Let us fix a complex number $\omega$ such that $\omega^2=i$. Then a $\bZ_4$-grading on  $M_2(\C)$, denoted by $M_2(\C,\bZ_4)$ can be obtained, as follows. We set $\bZ_4\cong(\gamma)_4$. Then

\begin{displaymath}
M_2(\C,\bZ_4)=\langle I,C\rangle_e\oplus \langle \omega A, \omega B\rangle_\gamma\oplus \langle iI,iC\rangle_{\gamma^2}\oplus\langle \omega^3 A, \omega^3 B\rangle_{\gamma^3}.
\end{displaymath}

A $\z\times \bZ_4$-refinement of $M_2(\C,\bZ_4)$ is denoted by $M_2^{(8)}$.We set $\z\times \bZ_4\cong(\alpha)_2\times(\gamma)_4$. Then

\begin{displaymath}
M_2^{(8)}=\langle I\rangle_e\oplus \langle C\rangle_\alpha\oplus \langle \omega A\rangle_\gamma\oplus \langle \omega B\rangle_{\alpha\gamma}\oplus \langle iI\rangle_{\gamma^2}\oplus \langle iC\rangle_{\alpha\gamma^2}\oplus\langle \omega^3 A\rangle_{\gamma^3}\oplus \langle \omega^3 B\rangle_{\alpha\gamma^3}.
\end{displaymath}

We know that $\H\ot\H\cong M_2(\R)\ot M_2(\R)$. This isomorphism transfers the structure of graded division algebra from $\H^{(4)}\ot\H$ to $M_4(\R)\cong M_2(\R)\ot M_2(\R)$. This grading of $M_4(\R)$ is not a tensor product of gradings on the tensor factors $M_2(\R)$. Thus $R=M_4(\R)$ acquires a division $\zz\cong(\alpha)_2\times(\beta)_2$-grading $M_4^{(4)}$ whose components are as follows:
\begin{eqnarray}\label{eq12}
R_e&=&\langle I\ot I, C\ot I,A\ot C,B\ot C\rangle,\\R_\alpha&=&(I\ot C)R_e,\;R_\beta=(C\ot A)R_e,\,R_{\alpha\beta}=(C\ot B)R_e,\nonumber
\end{eqnarray}

Finally, given a nonsingular alternating bicharacter $\beta: G\times G\to\C$, $|G|=n^2$, we denote by $\cP(\beta)$ the unique, up to equivalence, fine grading on $M_n(\C)$ defined by $\beta$ (see the details in Section \ref{scxc} below). Now let us denote by $\cP(\beta)_\R$ the same grading, viewed as a grading of an algebra over $\R$. We call $\cP(\beta)_\R$ a \textit{Pauli grading}.

\begin{theorem}\label{tgdar} Any division grading on a real simple algebra $M_n(D)$, $D$ a real division algebra, is  equivalent to one of the following types
\begin{enumerate}
\item[$D=\R:$]
\begin{enumerate}
\item[\rm (i)] $(M_2^{(4)})^{\ot k}$;
\item[\rm (ii)] $M_2^{(2)}\otimes (M_2^{(4)})^{\ot (k-1)}$,  a coarsening of {\rm (i)};
\item[\rm (iii)]$M_4^{(4)}\ot(M_2^{(4)})^{\ot (k-2)}$,  a coarsening of {\rm (i)};
\end{enumerate}
\item[$D=\H:$]
\begin{enumerate}
\item[\rm (iv)] $\H^{(4)}\ot(M_2^{(4)})^{\ot k}$;
\item[\rm (v)] $\H^{(2)}\otimes (M_2^{(4)})^{\ot k}$ a coarsening of {\rm (iv)};
\item[\rm (vi)] $\H\ot(M_2^{(4)})^{\ot k}$, a coarsening of {\rm (v)};
\end{enumerate}
\item[$D=\C:$]
\begin{enumerate}
\item[\rm (vii)] $\C^{(2)}\ot(M_2^{(4)})^{\ot k}$,
\item[\rm (viii)] $\C^{(2)}\ot M_2^{(2)}\ot(M_2^{(4)})^{\ot (k-1)}$, a coarsening of {\rm (vii)};
\item[\rm (ix)] $\C^{(2)}\ot\H\ot (M_2^{(4)})^{\ot (k-1)}$, a coarsening of {\rm (vii)};
\item[\rm (x)] $ M_2^{(8)}\ot(M_2^{(4)})^{\ot (k-1)}$;
\item[\rm (xi)] $ M_2(\C,\ZZ_4)\ot(M_2^{(4)})^{\ot (k-1)}$, a coarsening of {\rm (x)};
\item[\rm (xii)] $M_2^{(8)}\ot M_2^{(2)}\ot (M_2^{(4)})^{\ot (k-2)}$, a coarsening of {\rm (x)};
\item[\rm (xiii)] $M_2^{(8)} \ot\H\ot (M_2^{(4)})^{\ot (k-2)}$, a coarsening of {\rm (x)};
\item[\rm (xiv)] Pauli gradings.
\end{enumerate}
\end{enumerate}
None of the gradings of different types or of the same type but with different values of $k$ is equivalent to the other.
\end{theorem}

Notice that (xii) is missing on the list in \cite{BZGDS} but appears in \cite{ARE}. It is useful to mention that the components of the gradings in  Theorem \ref{tgdar} are 1-dimensional in the cases (i), (iv), (vii) and (x). They are 2-dimensional in (ii), (v), (viii), (xi), (xii), (xiv) and 4-dimensional in the remaining cases (iii), (vi), (ix) and (xiii).

\section{Complex case}\label{scxc}

Any complex $G$-graded division algebra $R$ is isomorphic to a twisted group algebra $\C^\sigma G$ where $\sigma: G\times G\to\C^\times$ is a 2-cocycle on $G$, that is a map satisfying $\sigma(g,h)\sigma(gh,k)=\sigma(g,hk)\sigma(h,k)$, for any $g,h,k\in G$. Two twisted group algebras, corresponding to cocycles $\sigma$ and $\sigma'$ are isomorphic if and only if the cocycles are cohomologous, that is, there is a map $\alpha:G\to\C^\times$ such that $\sigma(g,h)=\sigma'(g,h)\alpha(g)^{-1}\alpha(h)^{-1}\alpha(gh)$, for all $g,h\in G$. If we set $\beta(g,h)=\dfrac{\sigma(g,h)}{\sigma(h,g)}$ then $\beta$ is a alternating bicharacter on $G$, and it does not depend on the choice of $\sigma$ in its cohomology class. If we denote by $X_g$ the element of $R=\C^\sigma G$ corresponding to $g$ in $\C G$ then, in $\C^\sigma G$, we have $X_gX_h=\beta(g,h)X_hX_g$.
An important observation is that knowing $\beta$ and the orders of elements of the group, completely defines $R$. To see this, let us write $G$ as the direct product of cyclic subgroups $G=(g_1)_{n_1}\times\cdots\times(g_m)_{n_m}$ of orders $n_1,\ldots, n_m$. Then consider an algebra $A$ given by $G$-graded generators $a_1,\ldots,a_m$ of $G$-degrees $g_1,\ldots, g_m$ and $G$-graded relations of two kinds: $a_1^{n_1}=1,\ldots,a_m^{n_m}=1$ and $a_ia_j=\beta(g_i,g_j)a_ja_i$. Clearly, $\dim A\le n_1\cdots n_m$. On the other hand, if $o(g)=n$ then in $R_g$ one can choose an element $u$ such that $u^n=z\in\C$; replacing $u$ by $v=\frac{1}{\sqrt[n]{z}}u$ we obtain an element $v$ in $R_g$ with $v^n=1$. Thus the generators $X_{g_1},\ldots,X_{g_m}$ of degrees $g_1,\ldots,g_m$ for $R$ can be chosen so that $X_{g_i}^{n_i}=I$. As a result, $A$ maps $G$-graded homomorphically onto $R$, and dim $R=n_1\cdots n_m$. Hence $R\cong A$.

 Let us denote such grading by $\mathcal{P}(\beta)$. Since in $\mathcal{P}(\beta)$, we always have \[X_gX_h(X_g)^{-1}(X_h)^{-1}=\beta(g,h)I,\] for any $g,h\in G$, it follows that $\mathcal{P}(\beta')$ is isomorphic to $\mathcal{P}(\beta)$ if and only if $\beta'=\beta$. In the case of equivalence (=weak isomorphism) $ \mathcal{P}(\beta')\sim\mathcal{P}(\beta)$, accompanied by a group automorphism $\alpha: G\to G$, we must have  $\beta'(\alpha(g),\alpha(h))=\beta(g,h)$, for all $g,h\in G$. In other words, $\mathcal{P}(\beta')$ is equivalent to $\mathcal{P}(\beta)$ if and only if, $\beta'$ belongs to the same orbit as $\beta$ under the natural action of $\Aut G$ on the set of alternating bicharacters $G\times G\to\C^\times$. For instance, if $\mathcal{P}(\beta)$ is a simple algebra then $\beta$ is nonsingular. As indicated in \cite[Chapter 2]{EK}, all non-singular alternating bicharacters on $G$ form one orbit, hence, given a finite abelian group $G$, all division $G$-gradings on $M_n(\C)$ are equivalent. Actually, such gradings exist if and only if $G\cong H\times H$ where $|H|=n$. 

In the case of commutative algebras, $A$ is just the group algebra $\C G$. So any commutative $G$-graded division algebra $R$ over $\C$ is isomorphic to the group algebra $\C G$. In other words, $R$ is isomorphic to the graded tensor product  $\C (g_1)_{n_1}\ot\cdots\ot\C (g_m)_{n_m}$ of group algebras of cyclic groups and is completely determined by the same invariants as the abelian group $G$.

\section{Pauli gradings}\label{spg}
Let  $\Gamma:R=\bigoplus_{g\in G}R_g$ be a division grading of a real algebra $R$ such that $\dim R_e=2$ and $R_e$ is a central subalgebra of $R$.  Then there is an isomorphism of algebras $\mu:\C\to R_e$. If $I=\mu(1)$ and $J=\mu(i)$ then setting $(a+bi)X=(aI+bJ)X$ where $a,b\in\R$, $X\in R$ turns $R$ to a complex algebra $R^\C$ endowed with a complex grading $\Gamma^\C$, which is a division grading. In this grading $(R^\C)_g=R_g$. By the previous section, there is an alternating bicharacter $\beta:G\times G\to\C^\times$ such that $\Gamma^\C\cong \cP(\beta)$. Since, obviously, $(\Gamma^\C)_\R\cong \Gamma$, we have that any real grading with two-dimensional central identity component is isomorphic to $\cP(\beta)_\R$, for an appropriate alternating complex bicharacter $\beta$ on $G$. We call the gradings of the type $\cP(\beta)$ \textit{Pauli gradings}.

Thus $R=\cP(\beta)_\R$, as a unital algebra with identity elements $I$, in terms of graded generators and defining relations can be given as follows. We choose the canonical decomposition $G=(g_1)_{n_1}\times\ld\times(g_m)_{n_m}$, $n_1 |\cdots|n_m$, and choose the generators $J$ with trivial grading and $X_i$ of degree $g_i$, for all $i=1,\ld,m$. Then we impose the relations $J^2=-I$, $X_i^{n_i}=I$, $JX_i=X_iJ$, for all $i=1,\ld,m$, and $X_iX_j=\mu(\beta(g_i,g_j))X_jX_i$. Here $\mu:\C\to \langle I,J\rangle$ is given, as above.

The first claim of the following theorem is clear from the above.

\begin{theorem}\label{Pauli} If $\Gamma: R=\bigoplus_{g\in G}R_g$ is a real $G$-graded division algebra such that $\dim R_e=2$ and $R_e$ is central then there exists an alternating bicharacter $\beta: G\times G\to\C^\times$ such that $\Gamma=\mathcal{P}(\beta)_\R$. Furthermore, $\mathcal{P}(\beta)_\R$ is isomorphic to $\mathcal{P}(\beta')_\R$ if and only if either $\beta'=\beta$ or $\beta'=\overline{\beta}$. Finally, $\mathcal{P}(\beta)_\R$ is equivalent to $\mathcal{P}(\beta')_\R$ if and only if  the orbit of $\beta'$ under the action of $\Aut G$ contains $\beta$ or $\overline{\beta}$.
\end{theorem}

\begin{proof} Suppose that $\vp:R=\cP(\beta)_\R\to\cP(\beta')_\R=R'$ is an isomorphism of real algebras such that for some automorphism $\alpha:G\to G$ we have $\vp(R_g)=R'_{\alpha(g)}$, for all $g\in G$. Clearly then $\vp(R_e)=R'_e$. If $\mu':\C\to R'_e$ is an isomorphism of algebras, similar to $\mu$ above, then $\wt{\vp}=(\mu')^{-1}\vp\mu:\C\to\C$ is an automorphism of $\C$ over $R$. Thus either $\wt{\vp}(z)=z$ or $\wt{\vp}(z)=\bz$, the complex conjugate of $z$. As a result, either $\vp(\mu(z))=\mu'(z)$ or $\vp(\mu(z))=\mu'(\bz)$. 

It remains, for any $g,h\in G$, to compute both sides of the equation  
\[
\vp(XYX^{-1}Y^{-1})=\vp(X)\vp(Y)\vp(X)^{-1}\vp(Y)^{-1},
\]
where $0\ne X\in R_g$ and $0\ne Y\in R_h$. Then $0\ne \vp(X)\in R_{\alpha(g)}$ and $0\ne \vp(Y)\in R_{\alpha(h)}$. Now the value of the left hand side is $\vp(\mu(\beta(g,h))$, while for the right one, we have $\mu'(\beta'(\alpha(g),\alpha(h))$. Using the above conclusion, in the first case, we have $\mu'(\beta(g,h))=\mu'(\beta'(\alpha(g),\alpha(h)))$, that is, $\beta=(\beta')^\alpha$. In the second case, $\mu'(\overline{\beta(g,h)})=\mu'(\beta'(\alpha(g),\alpha(h))$, that is, $\overline{\beta}=(\beta')^\alpha$.

Setting $\alpha=\id_G$, we obtain the ``only if'' part of conclusion of the Theorem about the isomorphism. If $\alpha$ is arbitrary, then we obtain the ``only if'' part of the conclusion about the equivalence. It remains to show that $R=\cP(\beta)_\R\cong R'=\cP(\overline{\beta})_\R$ (the case of equivalence is left to the reader as an exercise). 

Let us use for $R=\cP(\beta)_\R$ a presentation in terms of graded generators and defining relations, as described earlier. Let the presentation for $R=\cP(\overline{\beta})_\R$ be given by generators $J', X_i$, relations $(J')^2=-I$, $(X_i')^{n_i}=I$ and $X_i'X_j'=\mu'(\overline{\beta}(g_i,g_j))X_j'X_i'$, with indexes running through the same sets as in the case of $\cP(\beta)_\R$. If we consider the graded map $J\to -J$, $X_i\to X_i'$ then all the relations, with the possible exception of the last ones, obviously hold. To check this latter one, note that under our substitution, $\mu(z)$ is replaced by $\mu'(\bz)$. Then the relation of $\cP(\beta)_\R$: $X_iX_j=\mu(\beta(g_i,g_j))X_jX_i$ is replaced by 
$X_i'X_j'-\mu'(\overline{\beta(g_i,g_j)})X_j'X_i'$, which holds in $\cP(\overline{\beta})_\R$. Thus our map extends to a graded isomorphism of $\cP(\beta)_\R$ and $\cP(\overline{\beta})_\R$, because the dimensions of both algebras are the same number $2|G|$. Now the proof is complete.
\end{proof}

The study of non-singular alternating bicharacters over an algebraically closed field (often called the \textit{commutation factors}) is performed in the papers \cite{Zm} and \cite{Zo}.

 \section{Commutative case}\label{scc}
In the case of real commutative graded division algebras, the situation is different from the case of complex commutative graded division algebras. Given a natural number $n>1$ and a number $\ve=\pm 1$, we denote by $\cC(m;\ve)$ the real graded subalgebra in the complex group algebra $\C(g)_m$ of the cyclic group of order $m$ generated by $\mu g$, where $\mu^m=\ve$. Alternatively, one can view  $\cC(m;\ve)$ as a graded algebra generated by one element $x$ of degree $g$, with defining relation $x^m=\varepsilon\,I$. We will call this algebra a \textit{basic algebra of the first kind}. Of course, $\cC(2;-1)\cong\C^{(2)}$.

\begin{proposition}\label{tcc}Let $ A $ be a $G$-graded commutative division algebra over the field $ \mathbb{ R }$ of real numbers. Suppose $ \dim A_e=1 $. Let $G = (g_1)_{m_1}\times \cdots\times (g_k)_{m_k}$ be the direct product of primary cyclic subgroups of orders $m_ 1,\ldots, m_k $. Then
\begin{displaymath}
A\cong \cC(m_1;\ve_1)\ot\cdots\ot \cC(m_k;\ve_k),
\end{displaymath}
for a sequence of numbers $\ve_1,\ld,\ve_k=\pm 1$. Here, additionally, we can assume that $\ve_i=1$ in the case where $m_i$ is an odd number.
\end{proposition}
\begin{proof}
Clearly, $ \dim A= m_1\cdots m_k $. We choose in each $ A_{g_i} $ an element $\tilde {a}_i $. Then $\tilde {a}_i^{m_i}\in A_{g_i^{m_i}}=A_e $. If $ \lambda_i= \tilde{a}_i ^{m_i}> 0 $ and  either $\lambda_i>0$ or $m_i$ is odd,  then we can replace $ \tilde {a}_i $ by $ a_i=\frac{ 1 }{\sqrt [m_i]{\lambda_i} }\tilde {a}_i $ to have $ a_i^{m_i}=1$. If $ \lambda_i <0 $, $ m_i $ is even, all we can make is $ a_i^{m_i}=-1$. In this case, choose a polynomial ring $\mathbb{R}[x_1,\ld, x_k] $. Then the map $ \vp:x_i\to a_i $ extends to a homomorphism $ \vp$  of $\mathbb{R}[x_1,\ld, x_k] $ onto $ A $. The kernel  $K=\Ker\vp $ contains $ x_i^{m_i}-1 $ or $ x_i^{m_i}+1 $. Clearly, in this case, $ \mathbb{R}[x_1,\ld, x_k]/K $ is spanned by the images of the elements $ x_1^{s_1}\cdots x_k^{s_k} $, where $ 0 \le s_i< m_i $. So $\dim\mathbb{R}[x_1,\ld, x_k]/K\le m_1\cdots m_k $. Since $\varphi$ is an onto map, we have that $A\cong \mathbb{R}[x_1,\ld, x_k]/K$.

Now let us consider the complex group ring $\mathbb{C} G$. This is a naturally $G$-graded real division algebra. Assume that our relations are $ x_i^{m_i}-1 $ for $ 1\le i \le \ell $ and $ x_i^{m_i}+1 $ for $ \ell+1 \le i \le k $. Then there is a $G$-graded $\mathbb{R}$-subalgebra $ B$ generated by $g_1,\ldots g_{\ell},\omega_{\ell+1} g_ {\ell+1},\ldots, \omega_{k} g_ {k}$, where $\omega_j^{m_j}=-1 $. The dimension of $ B $ over $\mathbb{R}$ equals $ m_1\cdots m_k $. As a result,  $ A $ can be written as $ \R [x_1, \ldots, x_k]/(x_1^{m_1}-1,\ld, x_{\ell}^{m_{\ell}}-1, x_{\ell+1} ^{m_{\ell+1}}+1,\ld, x_k^{m_k}+1)\cong B\subset \mathbb{C} G$, where $ B=\alg\{g_1,\ld,g_{\ell},\omega_{\ell+1} g_ {\ell+1},\ldots, \omega_{k} g_ {k} \}$.

If $n_i$ is an odd number and $\ve_i=-1$ then replacing $a_i$ by $-a_i$ we will get $(-a_i)^{n_i}=1$.
\end{proof}

It follows from the above proposition that if we use a primary factorization of $G$ as the product of cyclic subgroups then we could write $G=H\times K$ where $H$ is a 2-subgroup of $G$ and $K$ is a subgroup of odd order. As a result,  any commutative graded division algebra $A$ with one-dimensional homogeneous components and support $G$ can be written as the graded tensor product $A\cong B\otimes \R K$, where $B$ is the commutative graded division algebra whose support is an abelian 2-group $H$ and $\R K$ is the real group algebra of $K$.

The ungraded structure  of real graded commutative division algebras is given by the following.

\begin{proposition}\label{rgcda}\emph{
\begin{itemize}
\item $\cC(2^q;1)\cong \R\ZZ_{2^q}\cong\C\op\cdots\op\C\op\R\op\R$, for $q\ge 1$,
\item $\cC(2^q;-1)\cong\C\op\cdots\op\C$, because $x^2=-1$ is not solvable in $\R$;
\item $\cC(2^{m_1};1)\ot\cdots\ot\cC(2^{m_t};1)\cong\C\op\cdots\op\C\op\underbrace{\R\op\cdots\op\R}_{2t}$,
\item $\cC(2^{m_1};\eta_1)\ot\cdots\ot\cC(2^{m_t},\eta_t) \cong\C\op\cdots\op\C$ if at least one $\eta_i=-1$.
\end{itemize}
}
\end{proposition}

We leave the proof to the reader as an easy exercise.

To determine the equivalence classes of real graded commutative division algebras, we need the following result.

\begin{lemma}\label{lce} Let $m,n$ be any natural numbers, $m\le n$. Then
\begin{displaymath}
\cC(2^m;-1)\ot\cC(2^n;-1)\sim\cC(2^m;1)\ot\cC(2^n;-1).
\end{displaymath}
\end{lemma}

\begin{proof}
 We write
\begin{eqnarray*}
&&\cC(2^{m};-1)\ot \cC(2^{n};-1)=\alg\{ u,v\,|\,u^{2^m}=-1,\; v^{2^n}=-1;uv=vu\}\\
&&\cC(2^{m};1)\ot \cC(2^{n};-1)=\alg\{ u_1,v_1\,|\,u_1^{2^m}=1,\; v_1^{2^n}=-1;uv=vu\}
\end{eqnarray*}
Here  $\cC(2^{m};-1)$ is graded by the cyclic group $G=(a)_{2^m}$ and $\cC(2^{n};-1)$ by the cyclic group $H=(b)_{2^n}$, so that $\cC(2^m;-1)\ot\cC(2^n;-1)$ is graded  by $G\times H=(a)_{2^m}\times(b)_{2^n}$. The same group grades $\cC(2^m;1)\ot\cC(2^n;-1)$.
Since $m\le n$,  then there is $k\le n$ such that $(v^{2^k})^{2^m}=-1$ and $(uv^{2^k})^{2^m}=1$. The mapping $u_1\to uv^{2^k}$, $v_1\to v$ extends to an ungraded algebra isoomorphism of $\cC(2^{m};1)\ot \cC(2^{n};-1)$ to $\cC(2^m;-1)\ot\cC(2^n;-1)$. Actually, this is an isomorphism, accompanied by an automorphism of $G\times H$, mapping $a\to ab^{2^k}$, $b\to b$. The proof is complete.
\end{proof}

Now we are ready to prove the main result about the structure of real commutative graded division algebras.
\begin{theorem}\label{tcca}
Any real finite-dimensional commutative graded division algebra is equivalent to exactly one of the following:
\begin{enumerate}
\item A naturally graded real group algebra $\R G$, for a finite abelian group $G$;
\item A naturally graded complex group algebra $\C G$, viewed as a real algebra, for a finite abelian group $G$;
\item A tensor product of graded algebras $\C(2^m;-1)\ot \R G$.
\end{enumerate}
\end{theorem}
\begin{proof}
Let $R$ be a real finite-dimensional commutative graded division algebra. We must have $\dim R_e=1$ or  $\dim R_e=2$. In the second case, $R_e\cong \C$ and $R$ is a complex graded division algebra. So this is $\C G$, for a finite abelian group $G$. Otherwise, we can use Proposition \ref{tcc} and the remark after the proof of that proposition to write $R$ as $\cC(2^{m_1};\ve_1)\ot\cdots\ot \cC(2^{m_k};\ve_k)\ot \R K$. Now $\cC(2^m;1)$ is just the real group algebra of the cyclic group of order $2^m$. So we may assume that all $\ve_i$ in the latter tensor product equal $-1$. Applying Lemma \ref{lce}, we can transform this tensor product to  $\cC(2^{m_1};1)\ot\cdots\ot \cC(2^{m_{k-1}};1)\ot\cC(2^{m_k};-1)\ot \R K$ and finally to the form $\C(2^m;-1)\ot \R G$. So we have proved that real finite-dimensional commutative graded division algebra is equivalent to one of the algebras in the statement of the Theorem.

Since the notion of equivalence includes an isomorphism of groups, $G$ is an invariant in the first two cases. Because the dimensions of homogeneous components in the first and the third case are 1 and in the second 2, none of the algebras of second case is equivalent to an algebra in the first or the third case. Since in the first case, in distinction with the third one, there are no homogeneous solutions to the equation $x^2=-1$, the algebras in the first case are not equivalent to the algebras in the third case. 

More precisely, if $1\le m< n$, then in  $\C(2^n;-1)$ there is a  homogeneous solution to the equation $x^{2^n}=-1$, whereas in $\C(2^m;-1)\ot \R G$, for any $G$,  such solutions do not exist. Indeed, a homogeneous element in $\C(2^m;-1)\ot \R G$ has the form of $\lambda x^s g$, where $\lambda\in \R$, $x$ a graded generator of $\C(2^m;-1)$, $s$ a natural number, $g\in G$. Then $(\lambda x^s g)^{2^n}=\lambda^{2^n} (x^s)^{2^n} g^{2^n}=\lambda^{2^n} (x^{2^n})^s g^{2^n}=\lambda^{2^n} g^{2^n}\ne -1$. As a result, the equivalent algebras in the third case must have isomorphic supports of the from $(a)_{2^m}\times G$. So the number $m$ and the group $G$ are defined uniquely. Now the proof is complete.
\end{proof}

For the proof of the ``nonequivalence'' part one could also use Proposition \ref{rgcda}. But the idea of looking at the homogeneous solutions of certain (sets of) equations suggested in the above proof will be used in considerably more complicated situations later in this paper.

\section{Noncommutative graded division algebras with 1-dimensional homogeneous components}\label{sOne_dim}

Let $\Gamma: R=\bigoplus_{g\in G}R_g$ be a division grading of an algebra $R$ over $\R$ by an abelian group $G$. For a subgroup $H$ of $G$ we denote by $R_H$ the sum of all graded components $R_g$, where $g\in H$. Then $R_H$ is a graded subalgebra of $R$. If $G=H\times K$ then, as a vector space, $R=R_H\otimes R_K$. If,   additionally, $R_H$ and $R_K$ commute, then $R\cong R_H\ot R_K$ is the tensor product of algebras, endowed with the tensor product of gradings.

Now let us assume $\dim R_e=1$. Let $G=(g_1)_{n_1}\times\cdots\times(g_q)_{n_q}$ be a primary  cyclic factorization of $G$, where each $g_i$ is an element of order $n_i$, $i=1,\ld,q$. Let $g,h\in \{ g_1,\ld,g_q\}$. Then for any nonzero $a\in R_g$ and $b\in R_h$ we should have $aba^{-1}b^{-1}=\lambda\in\R$. Since $a^n\in R_e\cong\R$, we must have $a^nba^{-n}=b=\lambda^n b$. Hence $\lambda^n=1$ and then also $\lambda=\pm 1$. If $n$ is odd then we must have $\lambda = 1$ so that $R_g$ is in the center of $R$. It is also possible that $R_g$ is in the center of $R$ for some more $g\in\{ g_1,\ld,g_q\}$. Let $H$ be the subgroup of $G$ generated by all these elements, $K$ the subgroup generated by the remaining elements in $\{ g_1,\ld,g_q\}$. Then $G=H\times K$ where $R_H$ is a commutative $H$-graded division algebra and $R_K$ is a $K$-graded division algebra, where $K$ is an abelian $2$-group. None of the graded components $R_{g_i}$, $g_i\in K$, is in the center of $R_K$.

The structure of $R_H$ has been described in Theorem \ref{tcca}. So from now on we assume that $G=(g_1)_{n_1}\times\cdots\times(g_q)_{n_q}$ is an abelian 2-group, $n_1|n_2|\ld|n_q$ and none of $R_{g_i}$ is central. Suppose that $k$ be the least number such that $R_{g_1}$ does not commute with $R_{g_k}$. If also $R_{g_1}$ does not commute with $R_{g_{t_1}},\ld,R_{g_{t_p}}$, for some $k<t_1<\cdots<t_p$, we replace $g_{t_j}$ by $g_kg_{t_j}$ and $R_{g_{t_j}}$ by $R_{g_kg_{t_j}}$. Since $o(g_{t_j})=o(g_kg_{t_j})$ and $R_{g_kg_{t_j}}=R_{g_k}R_{g_{t_j}}$, we will have that now $R_{g_1}$ commutes with all new $R_{g_j}$ but one, which is $R_{g_k}$. Now let $1<s_1<s_2<\ldots<s_p$ be such that $R_{g_k}$ does not commute with $R_{g_{s_j}}$. We replace $g_{s_j}$ by $g_1g_{s_j}$ and $R_{g_{s_j}}$ by $R_{g_1g_{s_j}}$. Then again $o(g_{s_j})=o(g_1g_{s_j})$ and now $R_{g_k}$ commutes with $R_{g_1g_{s_j}}=R_{g_1}R_{g_{s_j}}$. As  a result, we find that $R$ is a graded tensor product of graded subalgebras $R_{(g_1)\times(g_k)}$ and $R_K$, where $K=(g_2)_{n_2}\times\cdots\times(g_{k-1})_{n_{k-1}}\times(g_{k+1})_{n_{k+1}}\times\cdots\times(g_{q})_{n_{q}}$. This allows us to proceed by induction to finally write a graded tensor product $R=R_1\ot R_2\ot\cdots\ot R_n$, where each $R_i$ is a graded division algebra whose support is the product of at most two cyclic 2-groups.

Those algebras with cyclic support have been described in Proposition \ref{tcc}. A graded division algebra $R$ with $\dim R_e=1$, whose support is the product of two cyclic 2-groups $G=(g)_k\times(h)_\ell$  can be described, as follows. We choose $a\in R_g$ and $b\in R_h$ so that $a^k=\mu I$ and $b^\ell=\nu I$, where $\mu,\nu =\pm 1$. These elements are graded generators of the whole of  $R$, and they anticommute: $ab=-ba$. Since $R$ is graded, we have $\dim R= k\ell$. It remains to produce a $(k\ell)$-dimensional $G$-graded algebra with generators $u,v$ of the same degrees as $a,b$, respectively, satisfying the same relations. To do so, we consider $S=\R G\ot M_2(\C)$ and inside a subalgebra $\cD(k,\ell ;\mu,\nu)$ generated by $u= g\ot (\ve A), v= h\ot (\eta B)$. Here $\ve^k=\mu, \eta^\ell=\nu$ and $A$, $B$ standard Pauli matrices $A=\left(\begin{array}{cc} 1 &0\\0&-1\end{array}\right)$, $B=\left(\begin{array}{cc} 0 &1\\1&0\end{array}\right)$. The reader easily checks that the  required relations are satisfied and that the set of elements $u^iv^j=g^ih^j\ot \ve^i\eta^jA^iB^j$, where $0\le i<k$, $0\le j<\ell$ is linearly independent because this is true for the set of elements $g^ih^j$, $0\le i<k$, $0\le j<\ell$, forming a basis of $\R G$. So $R$ admits a graded homomorphism onto $S$. Comparing dimensions, we can see that $R\cong \cD(k,\ell ;\mu,\nu)$. As we see, if  $G=(g)_k\times(h)_\ell$, then, in terms of $G$-graded generators and defining relations, this algebra can be given by

 \begin{displaymath}
 \cD(k,\ell ;\mu,\nu)=(u,v\;|\; u^k=\mu I,\,v^\ell=\nu I, uv=-vu,\, \deg u=g,\,\deg v=h).
\end{displaymath}

Recall that $\mu,\nu=\pm 1$. In what follows we will always be assuming that $k\le\ell$.

We will call $ \cD(k,\ell ;\mu,\nu)$ a \textit{basic algebra of the second kind}. A generator $u$  is called \textit{even} if $\mu=1$. Otherwise, $u$ is called \textit{odd}. We write $d(u)=k$ and call $k$ the \textit{degree} of $u$. The same terminology will be used for $v$ and for the generator of the basic (commutative) algebra $\cC(m;\eta)$. The relations of the form $u^k=\mu I$ will be called the \textit{power relations} while those of the form $uv=\pm vu$ the\textit{ commutation relations}.

Notice that the algebras $R\cong \cD(k,\ell ;\mu,\nu)$ generalize both $\H^{(4)}$ and $M_2^{(4)}$. We have 
\begin{equation}\label{ehm} 
\H^{(4)}= \cD(2,2;-1,-1),\;M_2^{(4)}=\cD(2,2;1,1)\sim \cD(2,2;-1,1)\sim \cD(2,2;1,-1). 
\end{equation}
They can also be viewed as generalized Clifford algebras. 

Taking into account Theorem \ref{tcca}, we obtain the following.

\begin{theorem}\label{tonedim} Let $G$ be a finite abelian group. Then any real non-commutative finite-dimensional $G$-graded division algebra with one-dimensional graded components is equivalent to the graded tensor product of several copies of  $\cD(2^k,2^\ell;\mu,\nu)$, at most one copy of $\cC(2^m;-1)$ and a group algebra $\R H$, where $k,\ell,m$ are natural numbers, $k\le\ell$,  $\mu, \nu=\pm 1$ and $H$ a subgroup, which is a direct factor of $G$. 
\end{theorem}

\section{Equivalence classes of graded division algebras with 1-dimensional graded components}\label{ssiso_1}

\subsection{Basic algebras}

The case of basic algebras of the type $\cC(2^m;\eta)$ was discussed earlier in Theorem \ref{tcca}, we now discuss the question of  when $R_1=\cD(k_1,\ell_1;\mu_1,\nu_1)$ is equivalent to $R_2=\cD(k_2,\ell_2;\mu_2,\nu_2)$. Let us set $k_i=2^{r_i},\ell_i=2^{s_i}$, $i=1,2$.

Notice a technical remark.
\begin{lemma}\label{lsimple}
Let $u,v$ be two elements of an algebra $R$ such that $uv=-vu$. Then, for any natural $m,n,p$ we have
\begin{displaymath}
(u^mv^n)^p=u^{mp}v^{np}(-1)^{mn\frac{p(p-1)}{2}}.
\end{displaymath}
In particular, if one of  $m$ or $n$ is even or $p$ is divisible by $4$ then
\begin{displaymath}
(u^mv^n)^p=u^{mp}v^{np}.
\end{displaymath}
\end{lemma}
Back to the equivalence classes,  first of all, notice that a necessary condition for the equivalence is $r_1=r_2$ and $s_1=s_2$. Suppose that $\vp: R_1\to R_2$ is a homomorphism of algebras such that $\vp((R_1)_g)=(R_2)_{\psi(g)}$, where $\psi:G\to G$ is an automorphism of the group $G$. If $G=(a)\times (b)$, $o(a)=2^r$, $o(b)=2^s$, with $r=r_1=r_2$, $s=s_1=s_2$, then $\psi(a)=a^{\alpha_{11}}b^{\alpha_{12}}$, $\psi(b)=a^{\alpha_{21}}b^{\alpha_{22}}$. If we take generators $x_i\in (R_i)_{a_i}$, $y_i\in (R_i)_{b_i}$, then $\vp(x_1)=c x_2^{\alpha_{11}}y_2^{\alpha_{12}}$, $\vp(y_1)=dx_2^{\alpha_{21}}y_2^{\alpha_{22}}$. As noted earlier, $c,d=\pm 1$. Let us see, for what values of $\mu_2,\nu_2$ we can have such equations possible. We must have $\vp(x_1)^{2^r}=\mu_1 I$, $\vp(y_1)^{2^s}=\nu_1 I$, $\vp(x_1)\vp(y_1)=-\vp(y_1)\vp(x_1)$. So, we must have

\begin{displaymath}
(c x_2^{\alpha_{11}}y_2^{\alpha_{12}})^{2^r}=(x_2^{\alpha_{11}}y_2^{\alpha_{12}})^{2^r}=\mu_1\cdot I,\,(d x_2^{\alpha_{21}}y_2^{\alpha_{22}})^{2^s}=( x_2^{\alpha_{21}}y_2^{\alpha_{22}})^{2^s}=\nu_1\cdot I
\end{displaymath}
 and
\begin{displaymath}
(x_2^{\alpha_{11}}y_2^{\alpha_{12}})(x_2^{\alpha_{21}}y_2^{\alpha_{22}})=-(x_2^{\alpha_{21}}y_2^{\alpha_{22}})(x_2^{\alpha_{11}}y_2^{\alpha_{12}}).
\end{displaymath}
Performing computations, we find
\begin{displaymath}
x_2^{2^r\alpha_{11}}y_2^{2^r\alpha_{12}}(-1)^{\alpha_{11}\alpha_{12}\frac{2^r(2^{r}-1)}{2}}=\mu_1\cdot I,
\end{displaymath}
\begin{displaymath}
x_2^{2^s\alpha_{21}}y_2^{2^s\alpha_{22}}(-1)^{\alpha_{21}\alpha_{22}\frac{2^s(2^{s}-1)}{2}}=\nu_1\cdot I;
\end{displaymath}
\begin{displaymath}
x_2^{\alpha_{11}+\alpha_{21}}y_2^{\alpha_{12}+\alpha_{22}}(-1)^{\alpha_{12}\alpha_{21}}=-x_2^{\alpha_{11}+\alpha_{21}}y_2^{\alpha_{12}+\alpha_{22}}(-1)^{\alpha_{22}\alpha_{11}}.
\end{displaymath}
It follows from the last equation that the determinant $\Delta=\alpha_{11}\alpha_{22}-\alpha_{12}\alpha_{21}$ is an odd number. It follows from the second equation that
\begin{displaymath}
\nu_1=\mu_2^{2^{s-r}\alpha_{21}}\nu_2^{\alpha_{22}}(-1)^{\alpha_{21}\alpha_{22}2^{s-1}}.
\end{displaymath}
It follows from the first equation that $\alpha_{12}=2^{s-r}\beta_{12}$, for some integral $\beta_{12}$, and then
\begin{displaymath}
\mu_1=\mu_2^{\alpha_{11}}\nu_2^{\beta_{12}}(-1)^{\alpha_{11}\alpha_{12}2^{r-1}}.
\end{displaymath}

Let us first assume that $r<s$. In this case, the numbers $2^{s-1}$,  $2^{s-r}$ and $\alpha_{12}=2^{s-r}\beta_{12}$ are even. Since $\alpha_{11}\alpha_{22}-\alpha_{12}\alpha_{21}$ is an odd number, it also follows that $\alpha_{11}$ and $\alpha_{22}$ are odd numbers. As a result, our two latter equations simplify and give $\nu_1=\nu_2$ and $\mu_1=\mu_2\nu_2^{\beta_{12}}$ . Since $\beta_{12}$ can be both even or odd, this tells us that, in this case, there are three equivalence classes:
\begin{displaymath}
\{\cD(2^r; 2^s,1;1)\}, \;\{\cD(2^r, 2^s;-1,1)\},\;\{\cD(2^r; 2^s;-1,-1), \cD(2^r, 2^s;1,-1)\}.
\end{displaymath}

The next case is $2\le r=s$. In this case,
$2^{r-1}$ and $2^{s-1}$ are even, $\alpha_{12}=\beta_{12}$. So we have
\begin{displaymath}
\mu_1=\mu_2^{\alpha_{11}}\nu_2^{\alpha_{12}},\;\nu_1=\mu_2^{\alpha_{21}}\nu_2^{\alpha_{22}}.
\end{displaymath}
Clearly, the values in this equation depend only on the residue classes of $\alpha_{ij}$ $\mod 2$. Recalling that $\alpha_{11}\alpha_{22}-\alpha_{12}\alpha_{21}$ is an odd number, we have 6 options (6 being the number of nonsingular $2\times 2$-matrices over $\z$),  for the residue classes, hence for possible pairs: $\mu_1=\mu_2, \nu_1=\nu_2$; $\mu_1=\nu_2, \nu_1=\mu_2$; $\mu_1=\mu_2\nu_2, \nu_1=\nu_2$; $\mu_1=\mu_2, \nu_1=\mu_2\nu_2$; $\mu_1=\mu_2\nu_2, \nu_1=\mu_2$, and $\mu_1=\nu_2, \nu_1=\mu_2\nu_2$. 

As a result, we have two equivalence classes:
\begin{displaymath}
\{\cD(2^r, 2^r;1,1)\},\;\{\cD(2^r, 2^r;-1,-1), \cD(2^r, 2^r;1,-1),\cD(2^r, 2^r;-1,1)\}.
\end{displaymath}

One could also use the fact that there are only two orbits of the natural action of $\GL_2(\ZZ_2)$ on $\ZZ_2^2$.

Finally, the case $1=r=s$ is the case of Clifford algebras (see \cite{BZGDS}). In this case we have two classes:
\begin{displaymath}
\{\H^{(4)}\cong \cD(2, 2;-1,-1)\},\;\{M_2^{(4)}\cong\cD(2, 2;1,1), \cD(2, 2;1,-1),\cD(2, 2;-1,1)\}.
\end{displaymath}

\begin{proposition}\label{psingle_1} The following are the equivalence classes of algebras of the form $\cD(2^r, 2^s;\mu,\nu)$
\begin{enumerate}
\item$ \{\H^{(4)}\cong \cD(2, 2;-1,-1)\}$;
\item $\{M_2^{(4)}\cong\cD(2, 2;1,1), \cD(2, 2;1,-1),\cD(2, 2;-1,1)\}$;
\item $\{\cD(2^r, 2^r;1,1)\}$, where $r>1$;
\item $\{\cD(2^r, 2^r;-1,-1), \cD(2^r, 2^r;1,-1),\cD(2^r, 2^r,-1,1)\}$, where $r>1$;
\item $\{\cD(2^r, 2^s;1,1)\}$, where $1\le r<s$;
\item $\{\cD(2^r, 2^s;-1,1)\}$, where $1\le r<s$;
\item $\{\cD(2^r, 2^s;-1,-1), \cD(2^r, 2^s;1,-1)\}$, where $1\le r<s$.
\end{enumerate}
\end{proposition}

\subsection{Tensor products of basic algebras. Commutation relations} A much more complicated case is the equivalence of the \textit{tensor products} of  algebras of the type $\cD(k,\ell;\mu,\nu)$ or $\cC(m;\eta)$. It was mentioned in \cite{BZGDS} that $\H^{(4)}\ot\H^{(4)}$ is equivalent to $M_2^{(4)}\ot M_2^{(4)}$. In other words, $\cD(2, 2;-1,-1)\ot  \cD(2, 2;-1,-1)$ is equivalent to $\cD(2, 2;1,1)\ot  \cD(2, 2;1,1)$. Also, $\H^{(4)}\ot\C^{(2)}$ is equivalent to $M_2^{(4)}\ot\C^{(2)}$. In other words, $\cD(2, 2;-1,-1)\ot  \cC(2; -1)$ is equivalent to $\cD(2, 2;1,1)\ot  \cC(2; -1)$. Thus, an additional work is necessary to determine when two tensor products of gradings mentioned in Theorem \ref{tonedim} are equivalent to each other.

\medskip

Before we proceed with our classification, let us remark that if
 \begin{displaymath}
G=G_1\times\cdots\times G_k
\end{displaymath}
is the factorization of the group $G$ through its Sylow subgroups, then
 \begin{displaymath}
R\cong R_{G_1}\ot\cdots\ot R_{G_k}.
\end{displaymath}
Moreover, if $R\sim R'$, where
 \begin{displaymath}
R'\cong R_{G_1}'\ot\cdots\ot R_{G_k}',
\end{displaymath}
then $R_{G_1}\sim R_{G_1}',\ld,R_{G_k}\sim R_{G_k}'$. At the same time, if the order of $G_i$ is odd then $R_{G_i}\sim \R G_i$. \textit{This enables us to restrict ourselves to the case, where the support of the grading is an abelian 2-group.}
Now each tensor product
\begin{equation}\label{tpDD}
R=\cD(2^{k_1},2^{\ell_1};\mu_1,\nu_1)\ot\cdots\ot\cD(2^{k_s},2^{\ell_s};\mu_s,\nu_s)\ot \cC(2^{m_1};\eta_1)\ot\cdots\ot\cC(2^{m_t};\eta_t)
\end{equation}
gives rise to a sequence 
\begin{equation}\label{ech}
\chi=\{(k_1,\ell_1;\mu_1,\nu_1),\ldots,(k_s,\ell_s;\mu_s,\nu_s), (m_1;\eta_1),\ldots,(m_t;\eta_t)\}. 
\end{equation}
Such sequence will be called the \textit{characteristic} of $R$ and denoted by $\chi (R)$.
We will also define the \textit{truncated characteristic} of $R $ by setting
\begin{equation}\label{etch}
\bch (R)=\{(k_1,\ell_1),\ld,(k_s,\ell_s),m_1,\ld,m_t\}. 
\end{equation}

Since permuting  the components in the characteristic does not change the equivalence class of a related graded division algebra, the equality of two characteristics will be always understood up to a permutation of its tuples. 

\begin{proposition}\label{pchar_1}
If $R_1$ and $R_2$ are equivalent graded division algebras of the form (\ref{tpDD}), then $\bch(R_1)=\bch(R_2)$.
\end{proposition}

\begin{proof}
Without  loss of generality, we may assume that both $R_1$ and $R_2$ are graded by the same universal abelian 2-group $G$. In other words, as real vector spaces, $R_1$ and $R_2$ are isomorphic to the group algebra $\R G$. Let $\{ X_g\:|\:g\in G\}$ be a basis of this vector space and $\beta_1$, respectively, $\beta_2$ the commutation factors for $R_1$, respectively $R_2$. That is, $X_gX_h=\beta_i(g,h)X_hX_g$ in $R_i$, $i=1,2$.  Clearly then that $\beta_1$ and $\beta_2$ take values  $\pm 1$ only. Moreover, $\beta_i(g,g)=1$, for all $g\in G$.
Now let $\vp:R_1\to R_2$ be an algebra isomorphism and $\alpha:G\to G$ a group automorphism such that $\vp((R_1)_g)= (R_2)_{\alpha(g)}$. As previously, we must have $\beta_2(\alpha(g),\alpha(h))=\beta_1(g,h)$, for all $g,h\in G$. In other words, $\beta_1$ and $\beta_2$ must be in the same orbit under the natural action of $\Aut G$ on the group of (skew)symmetric bicharacters of $G$.

In the case where $G$ is an elementary abelian 2-group, one can view $G$ as a vector space over $\z$ and a commutation factor $\beta(g,h)=(-1)^{\gamma(g,h)}$ where $\gamma$ is a (skew)symmetric bilinear function with values in $\z$. Also, $\alpha$ is a linear transformation of $G$, which naturally acts on bilinear functions. So a necessary condition for $R_1$ to be equivalent to $R_2$ is that $\rank \gamma_1=\rank\gamma_2$, where $\beta_i(g,h)=(-1)^{\gamma_i(g,h)}$, $i=1,2$. In this case, both  $\bch(R_1)$ and $\bch(R_2)$ have the form of $\{\underbrace{(1,1),\ld,(1,1)}_m,1,\ld,1)\}$, where $2m=\rank \gamma_1=\rank\gamma_2$.

In the case where $G$ is not an elementary abelian 2-group, we denote by $G^2$ the subgroup of the squares of the elements in $G$. Clearly, $\beta_i(g,G^2)=1$, $i=1,2$. This allows one to define bicharacters $\beta_i^1: G/G^2\times G/G^2\to \{\pm 1\}$ by setting $\beta_i^1(gG^2,hG^2)= \beta_i(g,h)$. If $\alpha\circ\beta_1=\beta_2$ then $\alpha^1\circ\beta_1^1=\beta_2^1$, where $\alpha^1(gG^2)=\alpha(g)G^2$. This completely defines the total number of pairs in $\bch(R_i)$.   

To obtain a more precise information,  denote by $G_k$ the subgroup of the elements of $G$ whose order divides $2^k$. Notice that for any $k$, $G_k$ is invariant under the automorphisms of $G$. Let $m$ be the least number such that $G=G_m$. We now can proceed with the proof of our proposition by induction on $m$. If $m=1$ then $G$ is elementary abelian and we are done. If $m>1$,we consider $G_{m-1}$. The subalgebras $R_1'$ of $R_1$ and $R_2'$ of $R_2$, whose support is $G_{m-1}$ are equivalent graded division algebras, the equivalence produced by the restrictions of $\vp$ to $R_1'$ and $\alpha$ to $G_{m-1}$.  The characteristics for $R_1'$ and $R_2'$ are obtained from 
those for $R_1$ and $R_2$, as follows. For instance, if $R=R_1$, $R'=R_1'$, consider (\ref{etch}). To obtain a characteristic for $R'$, we need to replace all entries of $m$ by $m-1$ and replace each $(k_i,m)$ and $(m,m)$ by $k_i,m-1$ and $m-1,m-1$, respectively. By induction, the truncated characteristics obtained in this way for $R_1'$ and $R_2'$ are the same, up to a possible permutation. That is, the set of pairs where the entries are less than $m$ must be the same, up to permutation, in $\bch(R_1)$ and $\bch(R_2)$. 

Let us assume that $R_1$ has $k_{ij}^{(1)}$ entries of  tensor factors  of the type $ \cD(2^i,2^j;\mu,\nu)$ and $\ell_i^{(1)}$ entries of  tensor factors  of the type $\cC(2^i;\eta)$. The numbers $k_{ij}^{(2)}$ and $\ell_i^{(2)}$ denote similar values for $R_2$. By induction, $k_{ij}^{(1)}=k_{ij}^{(2)}$, if $i,j<m$. Now the supporting grading group $G$ for $R_1$ takes  the two form
\begin{eqnarray*}
G\cong
\ZZ_{2^1}^{k_{11}^{(1)}+\sum_{j=1}^m k_{1j}^{(1)}+\ell_1^{(1)}}\times \ZZ_{2^2}^{k_{22}^{(1)}+\sum_{j=2}^m k_{2j}^{(1)}+\ell_2^{(1)}}\times\cdots\\ \times\ZZ_{2^{m-1}}^{k_{m-1,m-1}^{(1)}+\sum_{j=m-1}^m k_{m-1,j}^{(1)}+\ell_{m-1}^{(1)}}\times \ZZ_{2^m}^{2k_{mm}^{(1)}+\ell_m^{(1)}}.
\end{eqnarray*}

Similarly, 
\begin{eqnarray*}
G\cong
\ZZ_{2^1}^{k_{11}^{(2)}+\sum_{j=1}^m k_{1j}^{(2)}+\ell_1^{(2)}}\times \ZZ_{2^2}^{k_{22}^{(2)}+\sum_{j=2}^m k_{2j}^{(2)}+\ell_2^{(2)}}\times\cdots\\ \times\ZZ_{2^{m-1}}^{k_{m-1,m-1}^{(2)}+\sum_{j=m-1}^m k_{m-1,j}^{(2)}+\ell_{m-1}^{(2)}}\times \ZZ_{2^m}^{2k_{mm}^{(2)}+\ell_m^{(2)}}.
\end{eqnarray*}

So we have $m$ equations:
\begin{eqnarray*}
&&k_{11}^{(1)}+\sum_{j=1}^m k_{1j}^{(1)}+\ell_1^{(1)}=k_{11}^{(2)}+\sum_{j=1}^m k_{1j}^{(2)}+\ell_1^{(2)},\ldots,\\ && k_{m-1,m-1}^{(1)}+\sum_{j=m-1}^m k_{m-1,j}^{(1)}+\ell_{m-1}^{(1)}=k_{m-1,m-1}^{(2)}+\sum_{j=m-1}^m k_{m-1,j}^{(2)}+\ell_{m-1}^{(2)},\\&& 2k_{mm}^{(1)}+\ell_m^{(1)}= 2k_{mm}^{(2)}+\ell_m^{(2)}.
\end{eqnarray*}
However, using the induction hypothesis, we may write
\begin{eqnarray}\label{echar1}
k_{1m}^{(1)}+\ell_1^{(1)}&=&k_{1m}^{(2)}+\ell_1^{(2)},\ldots, k_{m-1,m}^{(1)}+\ell_{m-1}^{(1)}= k_{m-1,m}^{(2)}+\ell_{m-1}^{(2)},\\ 2k_{mm}^{(1)}+\ell_m^{(1)}&=& 2k_{mm}^{(2)}+\ell_m^{(2)}.\nonumber
\end{eqnarray}

Now we are going to compare the supports of the centers of $R_1$ and $R_2$. The centers are graded division algebras moved under equivalences to each other. The contribution of the center of  $\cD(2^k,2^\ell;\mu,\nu)$ with basis $\{x,y\}$ of degrees $a,b$, $(a)\cong\ZZ_{2^k}$, $(b)\cong\ZZ_{2^\ell}$ to the center of $R_1$ is the graded subalgebra generated by $x^2$, $y^2$ with the support $(a^2)\times (b^2)$.  At the same time, the tensor factors  of the type  $\cC(2^k;\eta)$ are central. Thus the support $H$ of the center of $R_1$, which is the same as for $R_2$, takes two following forms
\begin{displaymath}
H\cong
\ZZ_{2^1}^{ 2k_{22}^{(1)}+\sum_{j=2}^m k_{2j}^{(1)}+\ell_1^{(1)}}\times \ZZ_{2^2}^{2k_{33}^{(1)}+\sum_{j=3}^m k_{3j}^{(1)}+\ell_2^{(1)}}\times\cdots\times\ZZ_{2^{m-1}}^{2k_{m,m}^{(1)}+\ell_{m-1}^{(1)}}\times \ZZ_{2^m}^{ \ell_m^{(1)}},
\end{displaymath}
or
\begin{displaymath}
H\cong
\ZZ_{2^1}^{ 2k_{22}^{(2)}+\sum_{j=2}^m k_{2j}^{(2)}+\ell_1^{(2)}}\times \ZZ_{2^2}^{2k_{33}^{(2)}+\sum_{j=3}^m k_{3j}^{(2)}+\ell_2^{(2)}}\times\cdots\times\ZZ_{2^{m-1}}^{2k_{m,m}^{(2)}+\ell_{m-1}^{(2)}}\times \ZZ_{2^m}^{ \ell_m^{(2)}},
\end{displaymath}

So we can equate the exponents of primary cyclic factors in both factorizations. Again, since we proceed by induction on $m$, we arrive at the following equations

\begin{eqnarray}\label{echar2}
k_{2m}^{(1)}+\ell_1^{(1)}&=&k_{2m}^{(2)}+\ell_1^{(2)},\ldots, 2k_{m,m}^{(1)}+\ell_{m-1}^{(1)}= 2k_{m,m}^{(2)}+\ell_{m-1}^{(2)},\\ 
\ell_m^{(1)}&=& \ell_m^{(2)}.\nonumber
\end{eqnarray}

Now we compare  the systems of equation (\ref{echar1}) and (\ref{echar2}).  The last equation of  (\ref{echar2}) reads as $\ell_m^{(1)}=\ell_m^{(2)}$. Then from the last equation in (\ref{echar1}) we get  $k_{mm}^{(1)}=k_{mm}^{(2)}$. Using the second last equation in (\ref{echar2}), we then get $\ell_{m-1}^{(1)}= \ell_{m-1}^{(2)}$. The second last equation in  (\ref{echar1}) yields $k_{m-1,m}^{(1)}= k_{m-1,m}^{(2)}$. Proceeding in the same way, we finally get $k_{ij}^{(1)}=k_{ij}^{(2)}$, for all $1\le i,j\le m$ and  $\ell_j^{(1)}= \ell_j^{(2)}$, for all $1\le j\le m$.

Thus $\bch(R_1)=\bch(R_2)$ and the proof is complete.

\end{proof}

\subsection{Tensor products of basic algebras. Equivalence}\label{sse}

We keep assuming that the grading group of all graded division algebras is a finite abelian 2-group. Note that the statement of Theorem \ref{tonedim} can be made more precise, as follows. If among the generators of an algebra $R$ in  (\ref{tpDD}) there is a central odd generator whose degree is greater or equal than the degrees of all other odd generators then this algebra is equivalent  to an algebra where there is only one odd generator, and this is central. 
This can be done, using Lemma \ref{lce}, when we deal with the  tensor products of basic algebras of the first kind. 

At the same time, the argument of that lemma easily extends to the case where one of the algebras is of the first kind and the second of the second kind. To simplify this and many  further calculations, we use the following notation:
\begin{displaymath}
[\mu k,\nu \ell]=\cD(2^k,2^\ell;\mu,\nu),\: [\eta k]=\cC(2^k;\eta).
\end{displaymath}
For instance, we write $[1,-2]$ instead $\cD(2,4;1,-1)$ or $[-5]$ instead of $\cC(32;-1)$. We will also omit the sign of tensor product when using the above notation. Given the tensor product of two basic algebras $R\ot S$, we will denote by $\{u,v\}\{w,z\}$ the generating set for $R\ot S$, where $\{u,v\}$ is the canonical generating set of $R=\cD(2^k,2^\ell;\mu,\nu)$ and $\{w,z\}$ is the same for $S$. So we have $u^{2^k}=\mu I$, $v^{2^\ell}=\nu I$, $uv=-vu$. As usual, $R$ is naturally graded by $\ZZ_{2^k}\times\ZZ_{2^\ell}$. We have $uv=-vu$, $wz=-zw$ and the elements in $\{u,v\}$ commute with those in $\{w,z\}$. Likewise, if we deal with the tensor product $\cD(2^k,2^\ell;\mu,\nu)\ot\cC(2^m;\eta)=[\mu k,\nu\ell][\eta m]$, the canonical generating set will be $\{u,v\}\{w\}$, with the relations $u^{2^k}=\mu I$, $v^{2^\ell}=\nu I$, $w^{2^m}=\eta I$, $uv=-vu$, $uw=wu$, $vw=wv$.

Let us check, for instance, that $[-k,\ell][-m]\sim[k,\ell][-m]$, provided that $k\le m$. Indeed if the canonical set of generators for $R=[-k,\ell][-m]$ is $\{x,y\}\{z\}$, $x^{2^k}=-I$, $y^{2^\ell}=I$, $z^{2^m}=-I$, $R_a=\R x$, $R_b=\R y$, $R_c=\R z$, $G=(a)_{2^k}\times(b)_{2^k}\times(c)_{2^m}$ then the new generators $x_1=xz^{2^{m-k}}$, $y_1=y$, $z_1=z$ whose gradings  $a_1=ac^{2^{m-k}}$, $b_1=b$ and $c_1=c$, are obtained from $a,b,c$ by an automorphism of the group $G$,  satisfy all the defining relations of $R_1=[k,\ell][-m]$ and thus provide us with a weak isomorphism of  $R_1$ and $R$. 

Thus, one of the cases to be considered in the classification of noncommutative graded division algebras with 1-dimensional components is where the algebras have the form of 
\begin{equation}\label{eCASE_1}
\cD(2^{k_1},2^{\ell_1};1,1)\ot\cdots\ot\cD(2^{k_s},2^{\ell_s};1,1)\ot \cC(2^m;-1)\ot\R H,
\end{equation}
 where $H$ is a direct factor of $G$ and $s,m\ge 1$. If two algebras of this kind are equivalent then the parameter $m$ must be the same because in an algebra with this parameter there is a central homogeneous solution of the equation $x^{2^m}=-1$ but there is no solution of $x^{2^n}=-1$, if $n>m$. Once $m$ is fixed, the isomorphism class of the group $H$ is also fixed. Since the truncated characteristic is an invariant by Proposition \ref{pchar_1}, it follows that $\{(k_1,\ell_1),\ldots,(k_s,\ell_s)\}$ is also uniquely defined, up to permutation.
 
Now if the highest degree of an odd generator in (\ref{tpDD}) does not appear among the central generators, then there is no need to consider central odd generators. The argument is essentially the same as above. For example, let us prove that $R=[k,-\ell][-m]\sim R_1= [k,-\ell][m]$ if $\ell>m$. If $\{x,y\}\{z\}$ is the canonical generating set for $R$. Then the element $y^{2^{\ell-m}}$ is a central element with grading $b^{2^{\ell-m}}$, so switching to $\{x,y\}\{y^{2^{\ell-m}}z\}$, provides us with the desired equivalence. As a result, in this case, we need to deal only with the algebras of the form
\begin{equation}\label{eCASE_2}
\cD(2^{k_1},2^{\ell_1};\mu_1,\nu_1)\ot\cdots\ot\cD(2^{k_s},2^{\ell_s};\mu_s,\nu_s)\ot\R H,
\end{equation}
where $\mu_i,\nu_i=\pm 1$ and $H$ is a direct factor of $G$.

But we are going to show that even these algebras admit a further significant reduction.

\begin{proposition}\label{pmu_1} Let $G$ be a finite abelian 2-group. Then any graded division algebra $R$ with 1-dimensional homogeneous components is equivalent to one of the following types:
\begin{enumerate}
\item[\rm Type 1] Algebras as in (\ref{eCASE_2}) where none of the parameters $\mu_i,\nu_i$ equals $-1$;
\item[\rm Type 2] Algebras as in (\ref{eCASE_1});
\item[\rm Type 3] Algebras as in (\ref{eCASE_2}) where exactly one of the parameters $\mu_i,\nu_i$ equals $-1$ and where there are no factors of the form $\cD(2,2;-1,1)$ or $\cD(2,2;1,-1)$;
\item[\rm Type 4] Algebras of the form $\cD(2,2;-1-1)\ot S$, where $S$ is of the Type 1.
\end{enumerate}
\end{proposition}

Recall that 
\begin{displaymath}
\cD(2,2;-1-1)\sim\H^{(4)},\; \cD(2,2;-1,1)\sim \cD(2,2;1,-1)\sim \cD(2,2;1,1)\sim M_2^{(4)}.
\end{displaymath}

\begin{proof} To prove this proposition, we need to prove that the number of odd generators in the tensor product of two basic algebras is at most 1, with the exception of Type 4.  The argument just before formula (\ref{eCASE_2}) works also in the case of two or more basic algebras of the second kind. Its outcome is that our algebra is equivalent to an algebra where all odd generators have the same degree.

As a result, we need to consider the product of two basic algebras with odd generators of the same degree $k$. Suppose first that $k>1$. Using Claim (4) in Proposition \ref{psingle_1}, we can assume that if we have a product of two basic algebras of the second kind, each with odd generators, and all generators have the same degree $k$, then, in fact, only one generator in each algebra is odd and the other is even. So we have to consider the (tensor) products $[-k,\ell][-k,m]$, where $k\le\ell\le m$ or $[\ell,-k][m,-k]$, where $\ell\le m\le k$, or $[-k,\ell][m,-k]$, where $m\le k\le\ell$, or $[\ell,-k][-k,m]$, where $\ell\le k\le m$ (this last case is symmetric with  the previous one). In all cases the procedure of reduction to one odd variable is similar. 

For instance, consider $R=[-k,\ell][m,-k]$, where $m\le k\le\ell$. We will  prove that $R\sim R_1=[-k,\ell][m,k]$. Let $\{u,v\}\{w,z\}$ be the canonical set of generators for $R$. Let $u\in R_a$, $v\in R_b$, $w\in R_c$ and $z\in R_d$. Here $o(a)=2^k$, $o(b)=2^\ell$, $o(c)=2^m$ and $o(d)=2^k$. We choose the new generating set $\{u,vw\}\{w,uz\}$ for $R$. Let us check the power relations: $u^{2^k}=-I$, $(vw)^{2^\ell}=(v^{2^\ell})(w^{2^\ell})=I$, because $d(w)\le d(v)$, $w^{2^m}=I$ and $(uz)^{2^k}=(u)^{2^k}(z)^{2^k}=(-I)(-I)=I$. Now we check the commutation relations: $u(vw)=-v(uw)=-v(wu)=-(vw)u$, $w(uz)=w(zu)=-z(wu)=-(zu)w=-(uz)w$. Finally, $uw=wu$, $u(uz)=(uz)u$, $(vw)w=w(vw)$ and  $(vw)(uz)=(vu)(wz)=(-(uv))(-(zw))=(uz)(vw)$. The gradings of the elements of the new generating set are $\{a, bc,c, ad\}$ such that $o(a)=2^k$, $o(bc)=2^\ell$, $o(c)=2^m$ and $o(ad)=2^k$. As a result, we see that $R\sim R_1$. Similar arguments work in the other cases. Thus, in the case $k>1$ the algebra $R$ is equivalent to an algebra of the type Type II.

Now suppose $k=1$. The argument just above works here except that we cannot assume that in each basic algebra only one generator is odd. The exceptional case is $\H^{(4)}=[-1,-1]$. First recall a well-known equivalence $\H^{(4)}\ot\H^{(4)}\sim M_2^{(4)}\ot M_2^{(4)}$. So $[-1,-1][-1,-1]\sim [1,1][1,1]$. One more case is $R=[-1,\ell][-1,-1]$, where $\ell>1$. We claim that in this case, $R\sim R_1=[-1,\ell][1,-1]$.  In this case, if $\{u,v\}\{w,z\}$ is the canonical set of generators for $R$, one has to choose the new generating set $\{u,vz\}\{uw,z\}$. Actually, since $[1,-1]\sim [1,1]$, we may say that $R\sim[-1,\ell][1,1]$. It follows that if one of the factors is $\H^{(4)}$ then the product is not an algebra of Type I or Type II only if it is of Type III.

With this, the proof of Proposition \ref{pmu_1} is now complete.

\end{proof}

\subsection{Tensor products of basic algebras. Nonequivalence}\label{ssne_1}

We know from Proposition \ref{pchar_1} that the truncated characteristics of two algebras with 1-dimensional components are the same. 
  
Our first goal in this section  is to prove  that the algebras of different types in Proposition \ref{pmu_1} are not equivalent.

 In our proofs, we will consider (systems) of equations of the form $x^2=\pm 1$ and compare the number of elements in the supports of the sets of homogeneous solutions. If these numbers for two algebras are different then the algebras are not equivalent. For instance, the following remark is very useful in the forthcoming arguments.

\begin{remark}\label{rpmu_1}
 If $k\ge 2$ then the equation $x^{2^k}=-1$ has no homogeneous solutions in the algebras of  Type 1 in Proposition \ref{pmu_1}. 
\end{remark}
 
The following is an easy exercise from the domain of Clifford algebras. 
 
\begin{lemma}\label{lpmu_1}
Let us denote by $d_{\pm}^m$ the number of elements in the support of the set of homogeneous solutions for $x^2=\pm 1$ in $[1,1]^{\ot m}$, $m\ge 1$.  We also set $d_+^0=1$ and $d_-^0=0$. Then, for all $m\ge 0$, we have
\begin{eqnarray}\label{e_Clifford}
d_+^{m+1}=3d_+^m+d_-^m, \;& d_-^{m+1}=d_+^m+3d_-^m\\
d_+^{m}=\frac{4^m+2^m}{2},\;& d_-^{m}=\dfrac{4^m-2^m}{2}
\end{eqnarray}
In particular, for each $m>0$, one of $d_+^m$, $d_-^m$ is congruent $0\!\! \mod 3$ while the other is congruent $1\!\! \mod 3$.  Also, $d_+^m > d_-^m$, for all $m\ge 0$. $\square$
\end{lemma}

By Proposition \ref{pmu_1}, we may assume that, in the characteristics of the algebras with 1-dimensional homogeneous components, that is, of the form (\ref{tpDD}), at most one parameter $\{\mu_1,\nu_1,\ld,\mu_s,\nu_s\}$ is $-1$, with the exception of $\cD(2,2;-1,-1)\ot S\sim\H^{(4)}\ot S$, where $\chi(S)$ has no $-1$ among these parameters.
 
For the statement of our results, we single out the algebras of the form
\begin{equation}\label{eCASE_2_1}
\cD(2^{k_1},2^{\ell_1};\mu_1,\nu_1)\ot\cdots\ot\cD(2^{k_s},2^{\ell_s};\mu_s,\nu_s).
\end{equation}. 
 
 If  $R$ is as in (\ref{eCASE_2_1}) and $\chi=\chi(R)$, we write $R=\cD(\chi)$. If $\chi$ has no $-1$ among the parameters $\mu_i,\nu_i$, we call $\chi$ \textit{even}. If $\chi$ has exactly one $-1$ then we call it \textit{odd}. Because of the equivalence (\ref {ehm}), the odd characteristics containing $(1,1;-1,1)$ or $(1,1;1,-1)$ are still called even.

Algebras of Type 1 are completely separated from each other by the parameter $\bch$, hence also by $\chi$.

\subsection{Case of odd generators of degree greater than 2}\label{sss2}

We first consider the case where we compare two algebras, as in (\ref{tpDD}), one of which satisfies an additional condition that there is an odd generator of degree greater than 2. This algebra cannot be an algebra of Type 1, as mentioned in Remark \ref{rpmu_1}.

\textbf{Step I.} Two  algebras are not equivalent if, say, in $R_1$ there is an odd generator of degree $2^k$, with $k\ge 2$ and in $R_2$ the degree of any odd generator is $2^\ell$, with $\ell<k$. This follows because in $R_1$ there is a homogeneous solution to $x^{2^k}=-1$ whereas in $R_2$ no such solution exist. Indeed, let $y= u_1^{\alpha_1}u_2^{\alpha_2}\cdots u_m^{\alpha_m}$ be such a solution where $u_1,u_2,\ld,u_m$ are some generators of $R_2$. If $u_i$ is an odd generator then $d(u_i)\,|\,2^{k-1}$. Now by Lemma \ref{lsimple},
 \begin{displaymath}
y^{2^k}=u_1^{2^k\alpha_1}u_2^{2^k\alpha_2}\cdots u_m^{2^k\alpha_m}=(u_1^{2^k})^{\alpha_1}(u_2^{2^k})^{\alpha_2}\cdots (u_m^{2^k})^{\alpha_m}=1.
\end{displaymath}

As a result, $R_1\not\sim R_2$. Additionally, if $p_1$ is the maximal number for $R_1$ such that there is a homogeneous $x$ with $x^{2^{p_1}}=-1$ then the similar number $p_2$ for $R_2$ must be the same, which follows because if, for instance $p_1>p_2$, then $x^{2^{p_1}}=-1$ can be solved for homogeneous elements of $R_1$ but not of  $R_2$. It follows that two algebras of Type 2 are equivalent if and only if their characteristics are the same.

Now if $R_1$ is of Type 2, $R_2$ is of Type 3, and $R_1\sim R_2$, both with odd generator of degree $m\ge 2$.  Then in $R_1$ we have a central homogeneous solution of $x^{2^m}=-1$ whereas $R_2$ does not have such solution.

It follows that we only need to consider pairs of algebras $R_1$ and $R_2$ such that $\mu_i^{(1)}=-1$ (or $\nu_i^{(1)}=-1$) in $R_1$, $\mu_j^{(2)}=-1$ (or $\nu_j^{(2)}=-1$) in $R_2$ and $k_i^{(1)}=k_j^{(2)}$ or $\ell_j^{(2)}$ (or $\ell_i^{(1)}=k_j^{(2)}$ or $\ell_j^{(2)}$). The notation is self-explanatory.

\textbf{Step II.} We adopt our notation from Proposition \ref{pmu_1}. We first consider a particular case of the product of two basic algebras. We consider three cases, as below. One easily checks that they expire all possibilities.

\textit{Case} ($a_1$) Let $R_1=[-k,\ell][k,\ell+p]$, $R_2=[k,\ell][-k,\ell+p]$, where $2\le k\le\ell$ and $p\ge 1$. Then, in $R_2$ we have a homogeneous solution of the system of equations $x^{2^k}=y^{2^\ell}=1$, $xy=-yx$. Meanwhile, in $R_1$ this system has no homogeneous solutions. Indeed, let $\{u_1,v_1\}\{u_2,u_2\}$ be the canonical generating set for $R_1$. Let $x=\pm u_1^\alpha v_1^{\alpha'} u_2^\beta v_2^{\beta'}$ is a solution in $R_1$ to the equation  { $x^{2^k}=1$ in $R_1$. Then, using Lemma \ref{lsimple},
\begin{displaymath}
1=u_1^{2^k\alpha} v_1^{2^k\alpha'} u_2^{2^k\beta} v_2^{2^k\beta'}=(-1)^\alpha  v_1^{2^k\alpha'}v_2^{2^k\beta'}
\end{displaymath}
It follows that $\alpha=2\kappa$, $\alpha'=2^{\ell-k}\rho$, $\beta'= 2^{\ell+p-k}\sigma$, hence
\begin{displaymath}
x=\pm u_1^{2\kappa} v_1^{2^{\ell-k}\rho} u_2^\beta v_2^{2^{\ell+p-k}\sigma}.
\end{displaymath}
Since also $y^{2^\ell}=1$, we must have $y=\pm u_1^{\pi'} v_1^{\rho'} u_2^{\beta''} v_2^{2^p\sigma'}$. It follows that $xy=yx$, a contradiction.

\textit{Case} ($a_2$) Here $R_1=[k,-\ell][k+p,\ell]$, $R_2=[k,\ell][k+p,-\ell]$, where $k+p\le\ell$, $p\ge 1$. Again, in $R_2$, there is a solution of the system $x^{2^k}=y^{2^\ell}=1$, $xy=-yx$. When we solve this system in $R_1$, we obtain
\begin{displaymath}
x=u_1^{\alpha}v_1^{2\beta} u_2^{2^p\gamma} v_2^{2\delta}\;, y=u_1^{\alpha'}v_1^{2\beta'} u_2^{\gamma'} v_2^{\delta'},
\end{displaymath}
where we have used the defining relation $v^{2^\ell}=-1$. Again, $xy=yx$, a contradiction.

\textit{Case} ($a_3$) Here $R_1=[-k,\ell][t,k]$, $R_2=[k,\ell][t,-k]$, where $t< k\le\ell$. Now we consider the system $x^{2^t}=y^{2^k}=1$, $xy=-yx$. This is solvable in $R_1$. If there is solution $x,y$ in $R_2$, then
\begin{displaymath}
x=u_1^{2\alpha}v_1^{2\beta} u_2^{\gamma} v_2^{2\delta}\;, y=u_1^{\alpha'}v_1^{2\beta'} u_2^{\gamma'} v_2^{2\delta'}.
\end{displaymath}
As before, we derive that $xy=yx$, a desired contradiction.

\textbf{Step III.} Now let $P$ and $Q$ be two algebras of the form (\ref{tpDD}). Suppose that $P$ is equivalent to $Q$ and $P, Q$  are of Type 3. In this case we can write $P=A\ot R_1$, $Q=A\ot R_2$, where $(R_1,R_2)$ is the pair of algebras considered on the previous Step II (one of the cases ($a_1$),
 ($a_2$), and ($a_3$)), and $A$ is of Type I.

As shown in each of those cases, there are $k,\ell$, where $\ell\ge k\ge  2$, such that the system of equations
\begin{equation}\label{e_plus}
x^{2^k}=y^{2^\ell}=1,\: xy=-yx
\end{equation}
has homogeneous solutions in $R_1$ but not in $R_2$. Let us compute the number of pairs $(g,h)\in G\times G$ such that there exists $x,y$, with $\deg x=g$, $\deg y=h$, satisfying (\ref{e_plus}) in $P$ and in $Q$.

Suppose $x^{2^k}=1$ in $P$. Then we can write $x=\bx\wx$, where $\bx\in A$, $\wx\in R_1$ and $\bx^{2^k}=\wx^{2^k}=1$. Similarly,  we can write $y=\by\wy$, where $\by\in A$, $\wy\in R_1$ and $\by^{2^\ell}=\wy^{2^\ell}=1$.

\begin{remark}\label{r_norm}
\begin{enumerate}
\item  Notice that the condition $xy=-yx$ is equivalent to the disjunction of two conditions: $\bx\,\by=\by\,\bx$ and $\wx\wy=-\wy\wx$ OR $\bx\,\by=-\by\,\bx$ and $\wx\wy=\wy\wx$;
\item If $(\bx\wx)^{2^k}=1$ then $ \bx^{2^k}=\alpha\cdot 1$ and $ \wx^{2^k}=\beta\cdot 1$. Since there are no homogeneous roots of $x^{2^k}=-1$ in $A$, it follows that $\alpha>0$. So an adjustment of $\bx$ and $\wx$ can be done, which allows us to assume $\alpha=\beta=1$.
\end{enumerate}
\end{remark}

Let us denote by $X_0$ the set of homogeneous $2^k$th roots of $1$ in $A$. We choose exactly a half $X=\{ x_1,\ld,x_M\}$ of $X_0$ so that if $a\in X$ then $-a\not\in X$. We also denote by $Y=\{ y_1,\ld,y_N\}$ the set of all homogeneous $2^\ell$th roots of 1 in $A$. Since $\ell\ge 2$, it follows that for any $x_i\in X$ it is true that either $x$ commutes with all elements of $Y$ or $x$ commutes with exactly a half of these elements and anticommutes with the remaining half of them. Indeed, if $x$ anticommutes with $y\in Y$ and commutes with $y_1,\ld,y_s$ then $yy_i\in Y$, for all $i=1,\ld,s$, and $x$ anticommutes with $yy_1,\ld,yy_s$. To check this one can use Lemma \ref{lsimple}.

Now let us write
\begin{displaymath}
X=\{x_1,\ld,x_{M_1}, x_{M_1+1},\ld,x_{M}\},\; M=M_1+M_2,
\end{displaymath}
where $x_1,\ld,x_{M_1}$ centralize $Y$ and each of $ x_{M_1+1},\ld,x_{M_2}$ does not commute with at least one of the elements in $Y$. Then the number of pairs $(g,h)$ in the support of $A$ such that there are elements $x\in A_g$, $y\in A_h$,  satisfying (\ref{e_plus}), equals
\begin{equation}\label{e_a1}
M_2\cdot \frac{N}{2}
\end{equation}
while the number of pairs $(g,h)$ in the support of $A$ such that there are elements $x\in A_g$, $y\in A_h$  satisfying
\begin{equation}\label{e_three_plus}
x^{2^k}=y^{2^\ell}=I,\; x\,y=y\,x
\end{equation}
equals
\begin{equation}\label{e_a2}
M_1N+M_2\cdot\frac{N}{2}=N\left(M_1+\frac{M_2}{2}\right).
\end{equation}

We denote by $a_0$ and $a_1$ the number of pairs $(g,h)\in G\times G$ such that there are $x\in A_g$ and $y\in A_h$ satisfying $x^{2^k}=y^{2^\ell}=1$ and $xy=yx$, $xy=-yx$, respectively.  It follows from (\ref{e_a1}) and (\ref{e_a2}) that $a_0>a_1$.

Now let  us consider a pair of algebras in Case  ($a_1$) of Step II and determine for them all similar values $M(R_i)$, $N(R_i)$, $M_1(R_i)$ and $M_2(R_i)$, where $i=1,2$.

The number of elements in the support of the set of homogeneous solutions of $x^{2^k}=1$ both in $R_1$ and in $R_2$ equals $2^{4k-1}$. Hence $M(R_1)=M(R_2)$.  It is shown in (II, $a_1$), $M_2(R_1)=0$, $M_2(R_2)>0$. 

The number of elements in the support of the set of homogeneous solutions of $y^{2^\ell}=1$ equals $2^{2k+2\ell}$, that is, $N(R_1)=N(R_2)$. If we denote by $b_0$, $b'_0$ the number of elements in the supports of the set of homogeneous solutions of (\ref{e_three_plus}) in $R_1$ and $R_2$ and by $b_1$, $b'_1$ the respective numbers for  (\ref{e_plus}), then, applying (\ref{e_a1})  and  (\ref{e_a2}), we obtain the following.
\begin{displaymath}
 b_1=0, b'_1>0, b_0+b_1=M(R_1)N(R_1)=M(R_2)N(R_2)=b_0'+b_1'.
\end{displaymath}

The total number of elements in the support of the set of homogeneous solutions of (\ref{e_plus}), according to Remark \ref{r_norm}, can be computed, as follows. In $AR_1$ this number equals $a_0b_1+a_1b_0=a_1b_0=a_1(b_0'+b_1')$. In $AR_2$ this number equals $a_0b'_1+a_1b'_0$. Since $a_0>a_1$, it follows that
\begin{displaymath}
a_0b'_1+a_1b'_0>a_1(b'_0+b'_1).
\end{displaymath}
Now we can see that $P=AR_1$ and $Q=AR_2$ are not equivalent.

Similarly, in the cases 
 (II, $a_2$) and (II, $a_3$)
 we have
\begin{displaymath}
M(R_1)=M(R_2),\; N(R_1)=N(R_2),\: M_2(R_1)=0,\;M_2(R_2)>0.
\end{displaymath}

Hence, $b_0+b_1=b_0'+b_1', b_1=0, b_1'\ne 0$ and $a_0b_1'+a_1b_0'>a_0b_1+a_1b_0$.

\medskip

{\sc Conclusion.} \textit{Two algebras $P$ and $Q$, as in (\ref{tpDD}), without odd generators of degree 2, are equivalent if and only if $\chi(P)=\chi(Q)$.}

\subsection{Case of odd generators of degree 2}\label{sss3}

It remains to consider the case of two algebras as  in (\ref{tpDD}), satisfying the following condition 

($\dagger$): if $\mu_i=-1$ (respectively, $\nu_i,\eta_i=-1
$) then $k_i=1$ (respectively, $\ell_i=m_i
=1$). 

Note that an algebra $R_1$ of Type 2 with this condition cannot be equivalent to an algebra $R_2$ of  any other type, satisfying ($\dagger$),  because in $R_1$ we have a central homogeneous solution of $x^2=-1$. In the algebras of other types, there could be solutions of this equation but they are not central. Since these other algebras are tensor products of $[k,\ell]$, $[-1,\ell]$, $[-1,-1]$, and $[m]$($=\cC(2^m;1))$, one can easily check the absence of central solutions in each of these algebras separately.

All the the remaining algebras belong to one of the following classes:
\begin{enumerate}
\item[(i)] Type 1 from Proposition \ref{pmu_1}
\item[(ii)] $\cD(2,2^{\ell_1};\mu_1,1)\ot \cdots\ot\cD(2,2^{\ell_s};\mu_s,1)\ot B$, where there is exactly one $\mu_i=-1$, hence by ($\dagger$) $\ell_i\ge 2$, and $B$ is as in (i), with all $k_i\ge 2$;
\item[(iii)] Type 4 from Proposition \ref{pmu_1}, that is, $\cD(2,2;-1,-1)\ot B$, where $B$ is in (i).
\end{enumerate}

If both $R$ and $R'$ belong to the same class (i) or (iii) then $R\sim R'$ if and only if  $\bch(R)=\bch(R')$ hence $\chi(R)=\chi(R')$.

Now suppose both $R$ and $R'$ are in (ii). In this case $R=R_1\ot A$, $R'=R'_1\ot A$ where
\begin{displaymath}
R_1=[-1,\ell][1,\ell+p][1,1]^{\ot t},\; R_1'=[1,\ell][-1,\ell+p][1,1]^{\ot t}\mbox{ where }p\ge 1,\:\ell\ge 2,
\end{displaymath}
and
$A$ is in class (i), without tensor factors $[1,1]$, in which case it does not have homogeneous solutions of $x^2=-1$.

As earlier, we will be counting the homogeneous solutions of (\ref{e_plus}) and (\ref{e_three_plus}) in $R_1$, $R_1'$, $R$ and $R'$. If $x^2=1$ and $x=\wx\bx$,
$\wx\in R_1$, $\bx\in A$ then $\wx^2=\bx^2=1$. Similarly, if $y^{2^\ell}=1$ and $y=\wy\by$, $\wy\in R_1$, $\by\in A$, then $\wy^{2^\ell}=1$ and $\by^{2^\ell}=1$.

As before, we denote by $a_0$ or $a_1$ the number of pairs $(g,h)\in G\times G$ such that there are $x\in A_g$ and $y\in A_h$ satisfying $x^2=y^{2^\ell}=1$ and $xy=yx$, in the case of $a_0$, or $xy=-yx$, in the case of $a_1$. By $b_0$ and $b_1$ (respectively, $b_0'$ and $b_1'$) we denote similar numbers for $R_1$ and $R_1'$.

Let us first evaluate $b_0$ and $b_1$. If $x^2=1$ and $Y$ denotes the set of all homogeneous solutions of $y^{2^\ell}=1$ then either $x$ commutes with all $y\in Y$ or with a half of these elements. Let us write
\begin{displaymath}
x=u_1^\alpha v_1^{2\beta}u_2^\gamma v_2^{2\delta}z,\; y=u_1^{\alpha'}v_1^{\beta'}u_2^{\gamma'} v_2^{2\delta'}z'.
\end{displaymath}

\smallskip

Here $\alpha,\beta,\gamma,\delta,\alpha',\beta',\gamma'\delta'$ are integers and $z,z'\in[1,1]^{\ot t}$. If $x$ commutes with all $y\in Y$, then $\alpha=0$, $z=1$.

The number of elements in the support of the set of homogeneous solutions of $x^2=1$, commuting with the whole of $Y$, equals 8 while the number of homogeneous solutions commuting with the half of $Y$ equals $8\cdot 4^t-8$. Setting $N=|Y|$, we obtain
\begin{displaymath}
b_0=8N+\frac{1}{2}N(8\cdot 4^t-8)=4N(4^t+1),\; b_1=\frac{1}{2}N(8\cdot 4^t-8)=4N(4^t-1).
\end{displaymath}
Using Remark \ref{r_norm}, we find that the number of pairs in the support of the set of homogeneous solutions of (\ref{e_plus}) in $R$ equals
 \begin{displaymath}
 a_0b_1+a_1b_0=4N(4^t-1)a_0+4N(4^t+1)a_1.
 \end{displaymath}
 Similarly, if $x^2=1$ in $R_1'$ then $x=u_1^\alpha v_1^{2\beta}u_2^\gamma v_2^{2\delta}z$ and the number of such $x$ commuting with $Y$ equals 4, so that
 \begin{displaymath}
 b_0'=2N(2\cdot 4^t+1),\; b_1'=2N(2\cdot 4^t-1).
 \end{displaymath}
 As a result, the number of pairs in the support of the set of homogeneous solutions of (\ref{e_plus}) in $R'$ equals
\begin{displaymath}
 a_0b_1'+a_1b_0'=2N(2\cdot 4^t-1)a_0+2N(2\cdot 4^t+1)a_1.
 \end{displaymath}
Hence the equality $a_0b_1+a_1b_0=a_0b'_1+a_1b_0'$ is possible only if $a_0=a_1$. But we know that $a_0>a_1$ and hence $R$ and $R'$ are not equivalent.

Now let us assume that $R$ is in Class (ii) while $R'$ in Class (i). Then
\begin{displaymath}
R=[-1,\ell]\ot[1,1]^{\ot t}\ot A,\; R'=[1,\ell]\ot[1,1]^{\ot t}\ot A,
\end{displaymath}
 where $t\ge 0$ and $ A$ of Type 1 in Proposition \ref{pmu_1}, without factors $[1,1]$.
 
Let us recall the numbers $d_{\pm}^t$, which are  the numbers of elements in the support of the set of homogeneous solutions for $x^2=\pm 1$ in $[1,1]^{\ot t}$. Let also $T$ denote the number of elements in the support of the set of homogeneous solutions for $x^2=1$ in $A$.
Then the number of elements in the support of the set of homogeneous solutions for $x^2=1$ is as follows. In $R$ it is equal to $\alpha=(2d_+^t+2d_-^t)T$, in $R'$ this is $\beta=4d_+^tT$. The equality $\alpha=\beta$ would mean that $d_+^t=d_-^t$. But it follows from (\ref{e_Clifford}) that  one of $d_+^t, d_-^t$ is divisible by 3 while the other is congruent $1\!\! \mod 3$. It follows that $R$ and $R'$ are not equivalent.

Next, let $R$ be in Class (i) and $R'$ in Class (iii). One can assume that 
\begin{displaymath}
R=[1,1]\ot[1,1]^{\ot t}\ot A,\; R'=[-1,-1]\ot[1,1]^{\ot t}\ot A,
\end{displaymath}
where $t\ge 0$ and  $A$ has no homogeneous roots of $x^2=-1$. The number of elements in the support of the set of homogeneous solutions of $x^2=1$ in $R$ equals $\alpha= (3 d_+^t+d_-^t)T$ and the same number for  $R'$ equals $\beta=(d_+^t+3d_-^t)T$, where $d_+^t, d_-^t$ and $T$ are the same as in the previous case. Then $\alpha=\beta$ is equivalent to $d_+^t=d_-^t$, which is not the case.

And the final case is where $R$ is as in (ii) while $R'$ is as in (iii). Then
\begin{displaymath}
R=[-1,\ell]\ot[1,1]^{\ot t}\ot A,\; R'=[1,\ell]\ot[-1,-1]\ot[1,1]^{\ot t-1}\ot A,
\end{displaymath}
where $\ell\ge 2$ and $t\ge 0$.
Using the same notation for $A$, $T$, $d_+^t$ and $d_-^t$, as earlier, if $t>1$, we obtain $\alpha=\alpha(R)=2T(d_+^t+d_-^t)=8(d_+^{t-1}+d_-^{t-1})T$ and $\beta=\beta(R')=4(d_+^{t-1}+3d_-^{t-1})T$. In this case, if we assume $\alpha=\beta$ then $d_+^{t-1}=d_-^{t-1}$.  So again, in this case, $R\not\sim R'$. Now if
\begin{displaymath}
R=[-1,\ell]\ot[1,1]\ot A,\; R'=[1,\ell]\ot[-1,-1]\ot A,
\end{displaymath} 
then the number of elements in the support of homogeneous solutions of $x^2=1$ in $R$ equals $(2\cdot 3+2\cdot 1)\cdot T=8T$, and for $R'$ this equals  $4\cdot T\cdot 3=12T$. Again, we have a contradiction.

As a result, we have the following.

\begin{proposition}\label{t_2_1}
Two graded division algebras $P$ and $Q$, of the form (\ref{tpDD}), and asuming the restrictions of Proposition \ref{pmu_1}, are equivalent if and only if $\chi(P)=\chi(Q)$.
\end{proposition}

\begin{theorem}\label{t_2_1_m} Any finite-dimensional noncommutative real graded division algebra with 1-dimensional components  is equivalent to exactly one algebra on the following list.
\begin{enumerate}
\item $\cD(\chi)\ot\R G$, where  $\chi $ can be even or odd and $G$ a finite abelian group;
\item $\cC (2^m;-1)\ot\cD(\chi)\ot\R G$, where  $\chi $ is even and  $G$ a finite abelian group;
\item $\H^{(4)}\ot\cD(\chi)\ot\R G$, where  $\chi$ is even and  $G$ a finite abelian group.
\end{enumerate}

\end{theorem}

 This Theorem together with Theorem \ref {tcca} completes the classification of real graded division algebras with one-dimensional homogeneous components

\section{Graded division algebras with noncentral 2-dimensional identity component}\label{s2nc}

Again, we have $R=\bigoplus_{g\in G}R_g$ but now $\dim R_e=2$ and $R_e\not\subset Z(R)$. As earlier, let
$G=(g_1)_{n_1}\times\ld\times(g_q)_{n_q}$ be a primary cyclic factorization of $G$, where each $g_i$ is an element of order $n_i$, $i=1,\ld,q$.

Since $R_e\cong\C$, there is an element $J\in R_e$ such that $J^2=-1$ and since $R_e$ is non-central, $J\not\in Z(R)$. In this case, for any $a\in R_g$, the map $x\to axa^{-1}$, where $x\in R_e$ is an automorphism of $R_e\cong\C$, hence $aJa^{-1}=\pm J$. Also, $a^2J=Ja^2$.

If $g\in G$ is an element of odd order $o(g)=2s-1$, then $(g^2)^s=g$. Given $a\in R_g$, we then have $(a^2)^s\in R_g$. Then $R_g=R_e (a^2)^s=\langle (a^2)^s, J(a^2)^s\rangle$. Since $a^2$ commutes with $J$, it follows that the whole of $R_g$ commutes with $J$. Thus $\bigoplus_{o(g)\mbox{ \tiny odd}}R_g$ commutes with $J$.

Therefore, only $a\in R_{g_i}$, with $o(g_i)$ a 2-power, can anticommute with $J$. Using the same method as in the proof of Theorem \ref{tonedim}, we can modify our primary cyclic factorization of $G$ so that
\begin{displaymath}
G=(g_1)_{n_1}\times (g_2)_{n_2}\times\cdots\times (g_p)_{n_p}\times (g_{p+1})_{n_{p+1}}\times\cdots\times (g_q)_{n_q}
\end{displaymath}
is such that $n_1$ and  $n_2|\ld|n_p$ are 2-powers, there is $a\in R_{g_1}$ such that $aJ=-Ja$ while the elements in $R_{g_2},\ld, R_{g_q}$ commute with $J$. We can assume that $n_1$ is the minimal number with this property.
Let us set $H=(g_2)_{n_2}\times\ld\times (g_q)_{n_q}$, $T=(g_1)_{n_1}$. We have that $Q=R_H$ commutes with $J$ while $P=R_T$ does not.

Clearly, $P$, as an algebra, is generated by $J$ and $a$. Let $S$ be the span of the set of homogeneous elements $x$ of $Q$ such that $xax^{-1}=\lambda a$, where $\lambda$ is a positive real number.  Notice that for such an element $x\in R_h$, if we take another $x'\in R_h$ with the same property, we must have $x'\in\R x$. Indeed, since $R_h$ commutes with $J$, the subspace $R_h$ is a 1-dimensional \textit{complex} space, $R_e\cong\C$ being the complex coefficients. So there is $z\in \C$ such that $x'=xz$.  Also, the conjugation of $R_e$ by $a$ is a complex conjugation in $R_e$. So if $x'a(x')^{-1}=\mu a$, where $\mu$ is a positive real number, then we must have $xzaz^{-1}x^{-1}=\mu a$, or $xz\bz^{-1}ax^{-1}=\mu a$. Hence $\frac{1}{|z|^2}z^2(xax^{-1})=\mu a$. Finally, we find that $z^2$ is a positive real number. It follows then that $z$ is a real number and so $x'\in \R x$. Also, if $x_1\in R_{h_1}$, $x_2\in R_{h_2}$ are such that $x_ia(x_i)^{-1}=\lambda_i a$, where $\lambda_i>0$, for $i=1,2$, then $(x_1x_2)a(x_1x_2)^{-1}=(\lambda_1\lambda_2)a$, proving that the space $S$ spanned by such homogeneous elements is a graded subalgebra with one-dimensional components.

Finally, notice that such $x$ exists in every homogeneous component of $Q$. Indeed if $x$ is an arbitrary element of $R_h$ then $xax^{-1}=ua$, for some $u\in R_e$. If we replace $x$ by $x'=xz$ then, as earlier, $x'a(x')^{-1}=\frac{1}{|z|^2}z^2(xax^{-1})=\frac{1}{|z|^2}z^2ua$. Taking $z$ such that $z^2=u^{-1}$, we will have $x'a(x')^{-1}=\frac{1}{|z|^2}a$ so that $x'$ is the desired element of  $R_h$.

As a result, $S$ is a  graded division subalgebra of $Q$ and $\dim S=|H|$.  We have described such subalgebras in Section \ref{sOne_dim}. Now we have two graded subalgebras in $R$: $P$ of dimension $2n_1$ and $S$ of dimension $|H|$. The bilinear map $f:P\times S\to R$ given by $f(a^ku,x)=a^kux$, where $u\in R_e$, $k=0,1,2,\ld n_1-1$, $x\in S$,  extends to the linear map $\bar{f}:P\ot S\to R$. Clearly, this is a map onto. Since the dimensions of $R$ and $P\ot S$ are the same, we have that $\bar{f}$ is a\textit{ vector space }isomorphism. If we prove that $S$ and $P$ commute, then $\bar{f}^{-1}$ is a $G$-graded algebra isomorphism $R\cong P\ot S$.

To prove that, indeed, $P$ and $S$ commute, we consider the action of $a$ on the complex space $Q$. Pick a homogeneous $x\in Q_h$ then $axa^{-1}=zx$ where $z\in R_e$. Since $a$ acts by conjugation on $R_e\cong \C$, we have that $a^2xa^{-2}=z\bz x=|z|^2x$. Now if $n_1=2m$ then $g_1^{2m}=e$, and so $(a^2)^m\in R_e$. Since $x$  commutes with $R_e$, we have that $(|z|^2)^mx=(a^2)^mx(a^{2})^{-m}=x$. It then follows that $|z|^{2m}=1$ and hence $|z|=1$. Thus $a^2$ commutes with the whole of $Q$. Since $a^2$ commutes with $a$ and $J$, it follows that $a^2$ is in the center of $R$.

Now let us take $x\in S_h$. There is positive real $\lambda$  such that $xax^{-1}=\lambda a$. Then $a^2=(xax^{-1})(xax^{-1})=\lambda^2a^2$ so that $\lambda=1$, as needed.

To determine a realization of $P$, let us  set $g=g_1,\, n=n_1$. We remember that $n$ is a 2-power. Then, for any homogeneous $a\in P$, we must have $a^n=\alpha I+\beta J\in R_e$. If $a\in P_g$ is such that also $aJ=-Ja$ then $a^n=\alpha I+\beta J$ so that, if we conjugate by $a$, we obtain $\alpha I-\beta J=\alpha I+\beta J$. As a result, $\beta=0$ and $a^n=\alpha I$, where $\alpha$ is a real number. If we replace $a$ by $\frac{1}{\sqrt[n]{|\alpha|}}a$, we will obtain $a^n=\pm I$.

Now let us take $M=\R(g)_n\ot M_2(\C)$, with natural grading by $(g)_n$. Then choose $u=1\ot C$ and $v= g\ot\omega A$, where $\omega$ is a complex number such that $\omega^n=\ve$, where $a^n=\ve I$. Also, $A$ and $C$ are standard Pauli matrices, as defined by (\ref{eE1}). Then consider the subalgebra $\wh{M}=\alg\{u,v\}$. Comparing defining relations and dimensions, we easily obtain $P\cong \wh{M}$ as $G$-graded algebras. Let us denote the algebra thus described by $\cE(n;\ve)$. This is  a\textit{ basic algebra of the third kind}.

\begin{theorem}\label{ttwodim}
Any division grading on a finite-dimensional real algebra with two-dimensional noncentral identity component is equivalent to the tensor product of one grading of the form $\cE(n;\ve)$, several gradings of the form  $\cD(k,\ell;\mu,\nu)$ and several gradings $\cC(m;\ve)$ where $m, n,k,\ell$ are natural numbers $>1$, $n,k,\ell$  are $2$-powers,  and $\ve, \mu, \nu=\pm 1$. Moreover, $n$ is the minimal positive integer such that the
equation $x^n=\pm 1$ has a homogeneous solution not lying in the centralizer of $R_e$.
\end{theorem}

\section{Equivalence of graded division algebras with noncentral 2-dimensional identity component}\label{sse2nc}

Clearly, $\cE(2^k;\mu)\sim\cE(2^\ell;\nu)$ if and only if  $k=\ell$ and $\mu=\nu$. Indeed, the first equation is guaranteed because the grading group of $\cE(2^k;\mu)$ is $\ZZ_{2^k}$. As for the second equation, if $k>1$  then an equation $x^{2^k}=-1$ has homogeneous solution in $\cE(2^k;-1)$ but not in $\cE(2^k;1)$. If $k=1$ then $\cE(2;-1)\sim\H^{(2)}$, while $\cE(2;1)\sim M_2^{(2)}$.

In what follows, we write $(\rho k]$ for $\cE(2^k;\rho)$, where $\rho=\pm 1$. In particular, $(\rho]$ stands for $(1]\sim M_2^{(2)}$ if $\rho=1$ and $(-1]\sim \H^{(2)}$ if $\rho=-1$. By $\{J,u\}$ we denote the standard generating set for $\cE(2^k;\mu)$, that is $u^2=-I$, $v^{2^k}=\mu I$ and $uv=-vu$. Now $\deg u=e,\deg v=g$, where $\ZZ_{2^k}=(g)_{2^k}$. We have the following equivalences for the tensor products  of basic algebras $\cE(2^k;\rho)=(\rho k]$ and $\cD(2^\ell, 2^m;\mu,\nu)=[\mu k,\nu \ell]$ or $\cC(2^\ell;\eta)=[\eta \ell]$. We write $\{v,w\}$ for the standard generating set of $[\mu k,\nu \ell]$ and $\{w\}$ for $[\eta \ell]$.

\begin{lemma}\label{leq_2}
\begin{enumerate}
\item $(-k][-\ell]\sim(k][-\ell]$ if $k\le \ell$, new generating set $\{J,uv^{2^{\ell-k}}\}\{v\}$;
\item $(-k][-\ell]\sim(-k][\ell]$ if $k>\ell$, new generating set $\{J,u\}\{u^{2^{k-\ell}}v\}$;
\item $(-k][\ell,-m]\sim(k][\ell,-m]$ if $k< m$, new generating set $\{J,uw^{2^{m-k}}\}\{v,w\}$;
\item $(-k][\ell,-m]\sim(-k][\ell,m]$ if $k> m$, new generating set $\{J,u\}\{v,u^{2^{k-m}}w\}$;
\item $(-k][\ell,-k]\sim(k][\ell,-k]$,  if $\ell\ge 2$, new generating set $\{J,uw\}\{Jv,w\}$;
\item $(-k][1,-k]\sim(k][1,-k]$ if $k\geq 2$, new generating set $\{J,uw\}\{Jvw^{2^{k-1}},w\}$;
\item $(-k][-\ell,m]\sim(-k][\ell,m]$ if $k>\ell$, new generating set $\{J,u\}\{u^{2^{k-\ell}}v,w\}$;
\item $(-k][-\ell,m]\sim(k][-\ell,m]$ if $k<\ell$, new generating set $\{J,uwv^{2^{\ell-k}}\}\{v,w\}$;
\item $(-k][-k,m]\sim(k][-k,m]$, if $m\ge 2$, new generating set $\{J,uv\}\{v,Jw\}$;
\item $(\rho k][-1,\nu\ell]\sim(\rho k][1,\nu\ell]$, if $\ell\le k$, new generating set $\{J,uw\}\{Jv,w\}$;\label{eq10}
\item $(\rho k][-1,-1]\sim(\rho k][1,-1]]$, if $k\ge 2$, new generating set $\{J,uw\}\{Jv,w\}$;
\item $(\rho][-1,-1]\sim(-\rho][1,-1]$, new generating set $\{J,uw\}\{Jv,w\}$.
\end{enumerate}
\end{lemma}

Thanks to this Lemma, tensor factors of Type 4 in Proposition \ref{pmu_1} do not appear in the following.

\begin{proposition}\label{p2list}
Any graded division algebra with noncentral 2-dimensional identity component is equivalent to one of the following
\begin{enumerate}
\item[$(1_2)$]  $\cE(2^k;1)\ot S$, where $S$ is of Type 1 in Proposition \ref{pmu_1};
\item[$(2_2)$]  $\cE(2^k;-1)\ot S$, where $S$ is of Type 1 in Proposition \ref{pmu_1};
\item[$(3_2)$] $\cE(2^k;1)\ot S$, where $S$ is of Type 2 or 3 in Proposition \ref{pmu_1}, without factors of the form $\cD(2,2^\ell;-1,1)$ with $\ell\le k$.
\end{enumerate}
\end{proposition}

Let $\{J,u\}$ be the canonical generating set of $(\mu k]= \cE(2^k;\mu)$, where $\mu=\pm 1$. So $Ju=-uJ$, $J^2=-I$, $u^2=\mu I$. Let $a$ be the generating element of the grading group $G\cong\ZZ_{2^k}$. Any nonzero homogeneous element of degree $a^s$ can be uniquely written in the form $zu^s$, where $z=\alpha I+\beta J$, where $\alpha,\beta\in\R$, with $\alpha^2+\beta^2\ne 0$. To facilitate further arguments, we do some calculations in $(\mu k]$. We set $\bz=\alpha I-\beta J$.

\begin{center}
$u^sz=\left\{ \begin{array}{ll}
zu^s&\mbox{ if } s\mbox{ is even}\\ \bz u^s&\mbox{ if } s\mbox{ is odd}
\end{array}
\right.$
\end{center}

\begin{center}
$(z_1u^r)(z_2u^s)=\left\{ \begin{array}{ll}
(z_1z_2)u^{r+s}&\mbox{ if } r\mbox{ is even}\\ z_1\bz_2 u^{r+s}&\mbox{ if } r\mbox{ is odd}
\end{array}
\right.$
\end{center}

If $1\le m\le k$,
\begin{center}
$(zu^s)^{2^m}=\left\{ \begin{array}{ll}
z^{2^m}u^{2^{m+1}t}&\mbox{ if } s=2t\\ (z\bz)^{2^{m-1}} u^{2^{m+1}t}u^{2^m}&\mbox{ if }s=2t+1
\end{array}
\right.$
\end{center}

In particular,
\begin{center}
$(zu^s)^{2^k}=\left\{ \begin{array}{ll}
z^{2^k}&\mbox{ if } s=2t\\ (z\bz)^{2^{k-1}} \mu&\mbox{ if }s=2t+1
\end{array}
\right.$
\end{center}

Let us assume $k\ge 2$ and determine the homogeneous solutions of the equations $x^{2^k}=\pm 1$ in $(\pm k]$.

\begin{lemma}\label{l2dim_calc}

The homogeneous solutions of the equation $x^{2^k}=1$ are as follows. 

In $(k]$ we have finitely many solutions of the form $zu^{2t}$, with $z^{2^{k}}=I$, and infinitely many solutions of the form $zu^{2t+1}$, where $|z|=1$.

In $ (-k]$ we have finitely many solutions of the form $zu^{2t}$, with $z^{2^t}=-I$. 

\medskip

The homogeneous solutions of the the equation $x^{2^k}=-1$ are as follows. 

In $(k]$ we have finitely many solutions of the form $zu^{2t}$, with $z^{2^{k}}=-I$.

In $ (-k]$ we have finitely many solutions of the form $zu^{2t}$, with $z^{2^t}=-I$ and infinitely many solutions of the form $zu^{2t+1}$, where $|z|=1$. 

\end{lemma}

\subsection{Nonequivalence}\label{ssNE_2}

In this section we will be dealing with the three types of algebras of Proposition \ref{p2list}. Since $2^k$ is always chosen to be the least possible degree of elements not commuting with $J$, we conclude that the number $k$ must be the same in two equivalent gradings. Now we are going to show that no algebra of the Type $(2_2)$ can be equivalent to an algebra of any of the Types $(1_2)$ or $(3_2)$. 

Suppose $R_1=(k]\ot S$, $R_2=(-k]\ot S'$ are of Types $(1_2)$, $(2_2)$, respectively. If $k\ge 2$ then the homogeneous solutions of $x^{2^k}=-I$ in each of these algebras are products of solutions of this equation in $(\pm k]$ and the finite number of solutions of $x^{2^k}=I$ in $S$ or $S'$. According to Lemma \ref{l2dim_calc}, the number of these solutions in $(k]$ is finite and in $(-k]$ infinite. Hence $R_1\not\sim R_2$.                                           

In the case $k=1$, we have factorizations 
\begin{displaymath}
R_1=(1]\ot [1,1]^{\ot s}\ot A
\end{displaymath}
\begin{displaymath}
R_2=(-1]\ot [1,1]^{\ot t}\ot B,
\end{displaymath}
where $A$ and $B$ have no homogeneous solutions of the equation $x^2=-1$. 

Once again, we recall the numbers $d_+^m$, (resp., $d_-^t$) of the elements in the support of the set of homogeneous solutions of
$x^2=1$ (resp., $x^2=-1$) in $[1,1]^{\ot m}$. We know from Lemma \ref{lpmu_1} that  $d_+^m > d_-^m$, for all $m$.  Let also $p$ and $q$ be the  numbers of elements in the supports of the sets of homogeneous solutions of $x^2=1$ in $A$ and $B$, respectively.  If $a_1$ and $b_1$ are the number of elements in the support of the set of homogeneous solutions of $x^2=-1$, $x^2=1$, respectively, in $R_1$, then
\begin{displaymath}
a_1=(d_+^s+2d_-^s)p,\quad b_1=(2d_+^s+d_-^s)p.
\end{displaymath}
This follows because $(1]$ is the graded division algebra $M_2^{(2)}$, see (\ref{eM2M4}); one immediately sees that the number of elements in the support of the set of homogeneous solutions of $x^2=-1$ is 1 and in the case of $x^2=1$ is 2. Similarly, in the case of $(-1]$, which is $\H^{(2)}$, the respective numbers are 2 for $x^2=-1$ and 1 for $x^2=1$. Then similar values $a_2,b_2$ for $R_2$ are
\begin{displaymath}
a_2=(2d_+^t+d_-^t)q,\quad b_2=(d_+^t+2d_-^t)q.
\end{displaymath}

If $R_1\sim R_2$ then $a_1=a_2, b_1=b_2$ and
\begin{displaymath}
\frac{p}{q}=\frac{d_+^t+2d_-^t}{2d_+^s+d_-^s}=\frac{2d_+^t+d_-^t}{d_+^s+2d_-^s},
\end{displaymath}
so that
\begin{equation}\label{equa2}
\frac{2d_+^t+d_-^t}{d_+^t+2d_-^t}=\frac{d_+^s+2d_-^s}{2d_+^s+d_-^s}.
\end{equation}
But, according to Lemma \ref{lpmu_1}, the left hand side of (\ref{equa2}) is greater than $1$ while the right hand side is less than $1$.

Now let $R_1$ be of the Type $(2_2)$ and $R_2$ of the Type $(3_2)$. Then
\begin{displaymath}
R_1=(-k]\ot [1,1]^{\ot t}\ot A,
\end{displaymath}
while $R_2$ is one of two forms
\begin{displaymath}
R_2'=(k]\ot [1,1]^{\ot s}\ot [-r,\ell] \ot B.
\end{displaymath}
or
\begin{displaymath}
R_2''=(k]\ot [1,1]^{\ot s}\ot [-r] \ot B.
\end{displaymath}
One more possible case would be
\begin{displaymath}
(k]\ot [1,1]^{\ot s}\ot [r,-\ell] \ot B,
\end{displaymath}
but this is quite similar to the first one and the reader might check this, as an exercise. 

In the above equations, $A$, $B$ have no homogeneous  solutions of $x^2=-1$. 

If $k\ge 2$, then the homogeneous solutions of $x^{2^k}=1$ in $R_1$ are products of the homogeneous solutions to this equation in $(-k]$ and $[1,1]^{\ot t}\ot A$. Using Lemma \ref{l2dim_calc}, we obtain that this number is finite. At the same time, the number of solutions of the same equation in $(k]$ is infinite. So it is infinite in each of the above forms of $R_2$, proving $R_1\not\sim R_2$.

In the case $k=1$, we denote by $d_\pm^t, d_\pm^s, p, q, a_1,b_1,a_2,b_2$ the same values as before. Denote also by $\widetilde d_+$
(resp., $\widetilde d_-$) the number of elements in the support of the set of homogeneous solutions of $x^2=1$ (resp., $x^2=-1$) in $[1,1]^{\ot s}\ot [-r,\ell]$. Then
\begin{displaymath}
\widetilde d_+ = 2d_+^s + 2d_-^s = \widetilde d_-
\end{displaymath}
and
\begin{displaymath}
a_1=(2d_+^t + d_-^t)p,\quad b_1=(d_+^t + 2d_-^t)p
\end{displaymath}
while
\begin{displaymath}
a_2=(2\widetilde d_+ +\widetilde d_-)q = (\widetilde d_+ +2 \widetilde d_-)q= b_2.
\end{displaymath}
But $a_1>b_1$ by (\ref{e_Clifford}), hence $R_1$ and $R_2'$ are not equivalent.

Now denote also by $\bd_+$
(resp., $\bd_-$) the number of elements in the support of the set of homogeneous solutions of $x^2=1$ (resp., $x^2=-1$) in $[1,1]^{\ot s}\ot [-r]$. Then
\begin{displaymath}
\bd_+ = d_+^s + d_-^s = \bd_-
\end{displaymath}
and
\begin{displaymath}
a_1=(2d_+^t + d_-^t)p,\quad b_1=(d_+^t + 2d_-^t)p
\end{displaymath}
while
\begin{displaymath}
a_2=(2\bd_+ +\bd_-)q = (\bd_+ +2 \bd_-)q= b_2.
\end{displaymath}
But $a_1>b_1$ by (\ref{e_Clifford}), hence $R_1$ and $R_2''$ are not equivalent.

Now we know that if $R=(\mu k]\ot S$, where $S$ is a graded division algebra with 1-dimensional graded components then the numbers $\mu$, $k$ as well as the type of $S$ are well-defined. It remains to show that the equivalence class of $S$ is well-defined, too.

Note that if there is an equivalence $\vp:R=P\ot S\to P\ot S'$, where $P$ is equivalent to $(\mu k]$, which is an identity map on $P$, then $S\sim S'$. Indeed, let $\{ J,b\}$ be the canonical generators of $P$; consider the centralizer $C$ of $b$ in $R$. This is a graded division algebra with 1-dimensional graded components, generated by central homogeneous element $b$ and $S$.  Applying our results about the equivalence of algebras with 1-dimensional graded components, we obtain $S\sim S'$. 

Assume $P=\alg\{J,b\}$,
$P'=\alg\{J,b'\}$. We know from the above, that the degree $n$ of $b'$ is the same as the degree of $b$ and the parity (that is, odd or even) must be the same, too.
Let $S=\alg\{u_1,v_1,u_2,v_2,\ld,u_s, v_s, w_1,w_2,\ld,w_r\}$, where
$u_i v_i=-v_iu_i$, for any $i=1,\ld,s$ while all the other generators commute. It follows from the proof of Theorem \ref{ttwodim} that in any homogeneous component there is a unique, up to a real multiple, element that commute with $b'$. Since
\begin{displaymath}
b'=zb^\delta u_1^{\alpha_1}v_1^{\beta_1}\cdots u_s^{\alpha_s}v_s^{\beta_s}w_1^{\gamma_1}\cdots w_r^{\gamma_r}
\end{displaymath}
where $z=p I+q J$, $p,q\in\R$, the element $u'_i$ from the same component as $u_i$ commuting with $b'$
is either equal to $u_i$ or to  $u_iJ$, depending on $\beta_i$. Similar statements holds for $v_i'$. Thus the commutation relations among new generators
$u_1', v_1',\ldots, u_s',v_s'$
 remain the same as among the old generators. At the same time, the power relation $u_i^{2^{k_i}}=\pm 1$ remains to be the same if $k_i>1$ but  can change to the negative if  $k_i=1$ (similarly, for $v_i'$). Since the degree of $b'$ is the same as that of $b$, it follows that if, say $\ell_i>n$ then $\beta_i$ must be even and so $u_i'=u_i$ satisfies the same power relation as $u_i$. Thus  the power relation can change only for the tensor factor $[\pm 1,\ell]$ where $\ell\le n$, $2^n=d(b)$. If $\ell_i\ge 2$ then applying item (\ref{eq10}) in Lemma \ref{leq_2} allows one to choose the generators with the same power relations.  
 
The parity of $v_i$ can change when we switch to $v_i'$ only if $\ell_i=1$. In this case, if we use the new generators $u_i'=u_iJ$ and $v_i'=v_iJ$ then we will replace $[\mu_i,\nu_i]$ by $[-\mu_i,-\nu_i]$. Note that thanks to our restrictions in Proposition \ref{p2list}, this could be only when $\mu_i=\nu_i=1$ (remember, $[-1,1]\sim[1,-1]\sim[1,1]$!). If $n\ge 2$, we still can use  Lemma \ref{leq_2} to return from $(\rho n][-1,-1]$ to $(\rho n][1,-1]\sim (\rho n][1,1]$. However, if $n=1$, Lemma \ref{leq_2} does not help. At the same time, $(b')^2=(bu_1v_1)^2=-\rho$!. As a reslut, there must be another $j\ne i$ such that the corresponding tensor factor is $[1,1]$ and $b'$ contains a factor of $u_jv_j$. In this case, after choosing $u_i'=u_iJ$, $v_i'=v_iJ$, $u_j'=u_jJ$, $v_j'=v_jJ$, with respect to these generators we will have $1,1][1,1]$ transforms to $[-1,-1][[-1,-1]$, which is equivalent to $[1,1][1,1]$.

Since, obviously, the generators $w_1,\ld,w_r$ can be left intact, our argument shows that we can choose basis generators $u_1'', v_1'',\ldots, u_s'',v_s'',w_1,\ldots,w_r$
satisfying the same relations as $u_1,v_1,\ldots,w_r$. As a result, the map
$\varphi$
\begin{displaymath}
b\to b', u_1\to u_1'', v_1\to v_1'',\ldots, w_r\to w_r''
\end{displaymath}
extends to a weak isomorpfism $R\to R$. Hence $\bS = \alg\{u_1'', v_1'',\ldots, u_s'',v_s'',w_1,\ldots,w_r\}$ is equivalent to $S$. Moreover, $R=P'\bar S$.

Finally, we have the following.

\begin{theorem}\label{teq_2_dim}
Every graded division algebra with noncentral 2-dimensional identity component is equivalent to exactly one of the following.

\begin{enumerate}
\item $\cE(2^k;1)\ot \cD(\chi)\ot\R G$, $\chi$ even and $G$ a finite abelian group;
\item $\cE(2^k;1)\ot \cD(\chi)\ot \cC(2^m;-1)\ot\R G$, $\chi$ even and $G$ a finite abelian group;
\item $\cE(2^k;-1)\ot \cD(\chi)\ot\R G$, $\chi$ even and $G$ a finite abelian group;
\item $\cE(2^k;1)\ot \cD(\chi)\ot\R G$, $\chi $ is odd, without quadruples of the form $(1,\ell;-1,1)$ with $\ell\le k$ and $G$ a finite abelian group;
\end{enumerate}
 
\end{theorem}

\medskip

\section{Graded division algebras with 4-dimensional identity component}\label{s4c}

Let $R=\bigoplus_{g\in G}R_g$ be a graded division algebra such that $R_e\cong\H$. It follows from the Double Centralizer Theorem \cite[Theorem 4.7]{J} that $R\cong R_e\ot C$ where $C$ is the centralizer  of $R_e$ in $R$. For the completeness, we give a proof, emphasizing the grading on the algebra. 

For each $g\in G$, there is $u_g\in R_g$ such that $R_g=R_eu_g$. The map $x\to u_gxu_g^{-1}$ is an automorphism of $R_e\cong \H$, so that there is $q\in R_e$ such that $u_gxu_g^{-1}=qxq^{-1}$. If we denote by $C_g$ the $g$th homogeneous component of the centralizer $C$ of $R_e$ in $R$, then $q^{-1}u_g$ is a nonzero element in $C_g$.  Now if we assume that  already $u_g\in C_g$, suppose that also $v\in C_g$. Since $R_g=R_eu_g$, we have that $v=qu_g$, for some $q\in R_e$. Then, for any $x\in R_e$, we should have $x=vxv^{-1}=qu_gxu_g^{-1}q^{-1}=qxq^{-1}$. Hence $q$ is a quaternion commuting with all other quaternions, that is, $q=\lambda\cdot 1$, where $\lambda\in \R$. Thus $C_g=\R u_g$. As a result, we have a graded division subalgebra $C=\bigoplus_{g\in G}C_g$ with one-dimensional components whose elements commute with the elements of $R_e\cong\H$. Since $R_g=R_eC_g$ we have that the linear span of the image of the bilinear map $(q,c)\to qc$ from $R_e\times C$ to $R$ is the whole of $R$. Thus we have an isomorphism of vector spaces $R_e\ot C\to R$ given by $q\ot c\to qc$. Since $C$ and $R_e$ commute, we have $R\cong \R_e\ot C\cong \H\ot C$. Note that we have described graded division algebras with 1-dimensional homogeneous components in Section \ref{sOne_dim}.

If $R_1$ and $R_2$ are two graded division algebras with 4-dimensional components then $(R_i)_e\cong\H$ and $R_i\cong R_e\ot C_{R_i}((R_1)_e)$,  for $i=1,  2$. If we set $A_i=C_{R_i}(R_e)$ then $A_i$ is a graded division algebra with 1-dimensional components, for $i=1, 2$. Let $f: R_1\to R_2$ be a weak isomorphism of graded divison algebras. Then $f((R_1)_e)=(R_2)_e$ and
\begin{displaymath}
f(A_1)=f(C_{R_1}((R_1)_e))= C_{f(R_1)}(f((R_1)_e))=C_{R_2}((R_2)_e)=A_2.
\end{displaymath}
Hence $A_1\sim A_2$.

\begin{theorem}\label{tMain}
If $R=\bigoplus_{g\in G} R_g$ is a finite-dimensional real graded division  algebra with $\dim R_e=4$ then $\R$ is equivalent to exactly one of the following
\begin{enumerate}
\item $\H\ot \cD(\chi)\ot\R G$, where  $\chi $ can be even or odd and $G$ a finite abelian group;
\item $\H\ot\cC (2^m;-1)\ot \cD(\chi)\ot\R G$, $\chi $ even and $G$ a finite abelian group;
\item $\H\ot \H^{(4)}\ot\cD(\chi)\ot\R G$, $\chi $ even and  $G$ a finite abelian group; 
\end{enumerate}
\end{theorem}

\end{document}